\documentclass[a4paper,12pt,leqno]{article}
\usepackage[utf8]{inputenc}

\usepackage[top=3cm, left=2.5cm, right=2.5cm, bottom=3cm]{geometry}

\usepackage{hyperref}
\hypersetup{colorlinks=true,allcolors=dodger}

\usepackage{amsmath,mathtools}
\usepackage{amsthm, pb-diagram}
\usepackage{amssymb,comment}
\usepackage{amsfonts,graphicx,color}
\usepackage{enumitem, fancyhdr, dsfont}
\usepackage[normalem]{ulem}
\usepackage{thmtools,cleveref}
\usepackage{stmaryrd}
\usepackage{fontawesome5}
\usepackage{lineno,multicol}

\usepackage{tikz,float}

\setlist{topsep=0ex}

\makeatletter
\def\namedlabel#1#2{\begingroup
    #2%
    \def\@currentlabel{#2}%
    \phantomsection\label{#1}\endgroup
}
\makeatother

    \DeclareMathOperator{\dom}{{\rm dom}}

    \newcommand{\Ewf}{\mathcal{E}}

    \newcommand{\Iwf}{\mathcal{I}}
    \newcommand{\Jwf}{\mathcal{J}}
    
    \newcommand{\Mwf}{\mathcal{M}}
    \newcommand{\Nwf}{\mathcal{N}}

    \newcommand{\Pwf}{\mathcal{P}}
    \newcommand{\Scal}{\mathcal{S}}
    \newcommand{\Swf}{\mathcal{S}}

    \newcommand{\bfrak}{\mathfrak{b}}
    \newcommand{\cfrak}{\mathfrak{c}}
    \newcommand{\dfrak}{\mathfrak{d}}

    \newcommand{\rfrak}{\mathfrak{r}}
    \newcommand{\sfrak}{\mathfrak{s}}
    \newcommand{\efrak}{\mathfrak{e}}

    \newcommand{\Bor}{\mathbb{B}}
    \newcommand{\Cbb}{\mathbb{C}}
    \newcommand{\Cor}{\mathbb{C}}
    
    \newcommand{\Dor}{\mathbb{D}}
    \newcommand{\Eor}{\mathbb{E}}

    \newcommand{\Mor}{\mathbb{M}}

    \newcommand{\Pbb}{\mathbb{P}}
    
    \newcommand{\Por}{\mathbb{P}}

    \newcommand{\LOCor}{\mathds{LOC}}

    \newcommand{\Mior}{\mathbb{MI}}
    \newcommand{\Pror}{\mathbb{PR}}
    \newcommand{\menos}{\smallsetminus}
    \newcommand{\vacio}{\varnothing}
    \DeclareMathOperator{\pts}{\mathcal{P}}

    \newcommand{\R}{\mathbb{R}}

    \newcommand{\la}{\langle}
    \newcommand{\ra}{\rangle}

    \DeclareMathOperator{\add}{\mathrm{add}}
    \DeclareMathOperator{\non}{\mbox{\rm non}}
    \DeclareMathOperator{\cov}{\mbox{\rm cov}}
    \DeclareMathOperator{\cof}{\mbox{\rm cof}}

    \DeclareMathOperator{\Seq}{\mathrm{seq}}

    \newcommand{\seq}[2]{\la #1 :\, #2\ra}
    \newcommand{\set}[2]{\{#1 :\, #2\}}

    \newcommand{\blc}{\mathfrak{b}^{\mathrm{Lc}}}
    \newcommand{\dlc}{\mathfrak{d}^{\mathrm{Lc}}}

    \newcommand{\Fr}{\mathrm{Fr}}

    \newcommand{\Lb}{\mathrm{Lb}}
    \newcommand{\gen}{\mathrm{gn}}

\newcommand{\baire}{\omega^\omega}
\newcommand{\cantor}{2^\omega}

\newcommand{\tilb}{\tilde b}
\newcommand{\largeset}[2]{\bigg\{#1:\; #2\bigg\}}

    \definecolor{carrotorange}{rgb}{0.93, 0.57, 0.13}
    \definecolor{dodger}{rgb}{0.0,0.5,1.0}
    
    \newcommand{\Swfw}{\Swf^{\mathrm{w}}}

\definecolor{sub0}{RGB}{29,32,137}
\definecolor{sub1}{RGB}{1,71,157}
\definecolor{sub2}{RGB}{1,104,183}
\definecolor{sub3}{RGB}{0,160,234}
\definecolor{sug}{RGB}{0,154,68}
\definecolor{suy}{RGB}{208,219,1}

\newcommand{\subiii}[1]{{\color{sub3}#1}}

\title{The cardinal characteristics of the ideal
generated by the $F_\sigma$ measure zero subsets of the reals}
\author{Miguel A.~Cardona%
\thanks{Email: \href{mailto:miguel.cardona@upjs.sk}{\texttt{miguel.cardona@upjs.sk}}\newline
}
}
\date{{\normalsize
${}^*$Institute of Mathematics, Pavol Jozef \v{S}af\'arik University,\\
041 80, Jesenn\'a 5, 040 01 Ko\v{s}ice, Slovakia\\\medskip
}
}

\begin{document}

\makeatletter
\def\@roman#1{\romannumeral #1}
\makeatother

\newcounter{enuAlph}
\renewcommand{\theenuAlph}{\Alph{enuAlph}}

\numberwithin{equation}{section}
\renewcommand{\theequation}{\thesection.\arabic{equation}}

\theoremstyle{plain}
  \newtheorem{theorem}[equation]{Theorem}
  \newtheorem{corollary}[equation]{Corollary}
  \newtheorem{lemma}[equation]{Lemma}
  \newtheorem{mainlemma}[equation]{Main Lemma}
  \newtheorem{prop}[equation]{Proposition}
  \newtheorem{clm}[equation]{Claim}
  \newtheorem{fct}[equation]{Fact}
  \newtheorem{question}[equation]{Question}
  \newtheorem{problem}[equation]{Problem}
  \newtheorem{conjecture}[equation]{Conjecture}
  \newtheorem*{thm}{Theorem}
  \newtheorem{teorema}[enuAlph]{Theorem}
  \newtheorem*{corolario}{Corollary}
  \newtheorem*{scnmsc}{(SCNMSC)}
\theoremstyle{definition}
  \newtheorem{definition}[equation]{Definition}
  \newtheorem{example}[equation]{Example}
  \newtheorem{remark}[equation]{Remark}
  \newtheorem{notation}[equation]{Notation}
  \newtheorem{context}[equation]{Context}
  \newtheorem{exer}[equation]{Exercise}
  \newtheorem{exerstar}[equation]{Exercise*}

  \newtheorem*{defi}{Definition}
  \newtheorem*{acknowledgements}{Acknowledgements}
  
\def\sectionautorefname{Section}
\def\subsectionautorefname{Subsection}

\maketitle

\begin{abstract}

Let $\mathcal{E}$ be the ideal generated by the $F_\sigma$ measure zero subsets of the reals. The purpose of this survey paper is to study the cardinal characteristics (the additivity, covering number, uniformity, and cofinality) of $\mathcal{E}$. 
\end{abstract}









\section{Introduction}

Let $\Mwf$ and $\Nwf$, as usual, denote the $\sigma$-ideal of first category subsets of $\R$ and the $\sigma$-ideal of Lebesgue null subsets of $\R$, respectively, and let $\Ewf$ be the ideal generated by the $F_\sigma$ measure zero subsets of $\R$. It is well-known that $\Ewf\subseteq\Nwf\cap\Mwf$. Even more, it was proved that $\Ewf$ is a proper subideal of $\Nwf\cap\Mwf$ (see~\cite[Lemma 2.6.1]{BJ}).

For $f,g\in\omega^\omega$ define
\[f\leq^*g\textrm{\ iff\ }\exists m<\omega\,\forall n\geq m\colon f(n)\leq g(n).\]
Let
\[\bfrak:=\min\{|F|:F\subseteq \omega^\omega\text{\ and }\neg\exists y\in \omega^\omega\, \forall x\in F\colon x\leq^* y\}\]
the \emph{bounding number}, and let
\[\dfrak:=\min\{|D|:D\subseteq \omega^\omega\text{\ and }\forall x\in \omega^\omega\, \exists y\in D\colon x\leq^* y\}\]
the \emph{dominating number}. As usual, $\cfrak:=2^\omega$ denotes the \emph{size of the continuum}.

Let $\Iwf$ be an ideal of subsets of $X$ such that $\{x\}\in \Iwf$ for all $x\in X$. Throughout this paper, we demand that all ideals satisfy this latter requirement. We introduce the following four \emph{cardinal characteristics associated with $\Iwf$}: 
\begin{align*}
 \add(\Iwf)&=\min\largeset{|\Jwf|}{\Jwf\subseteq\Iwf,\,\bigcup\Jwf\notin\Iwf},\\
 \cov(\Iwf)&=\min\largeset{|\Jwf|}{\Jwf\subseteq\Iwf,\,\bigcup\Jwf=X},\\
 \non(\Iwf)&=\min\set{|A|}{A\subseteq X,\,A\notin\Iwf},\textrm{\ and}\\
 \cof(\Iwf)&=\min\set{|\Jwf|}{\Jwf\subseteq\Iwf,\ \forall\, A\in\Iwf\ \exists\, B\in \Jwf\colon A\subseteq B}.
\end{align*}
These cardinals are referred to as the \emph{additivity, covering, uniformity} and \emph{cofinality of $\Iwf$}, respectively. The relationship between the cardinals defined above is illustrated in \autoref{diag:idealI}.

\begin{figure}[ht!]
\centering
\begin{tikzpicture}[scale=1.5]
\small{
\node (azero) at (-1,1) {$\aleph_0$};
\node (addI) at (1,1) {$\add(\Iwf)$};
\node (covI) at (2,2) {$\cov(\Iwf)$};
\node (nonI) at (2,0) {$\non(\Iwf)$};
\node (cofI) at (4,2) {$\cof(\Iwf)$};
\node (sizX) at (4,0) {$|X|$};
\node (sizI) at (5,1) {$|\Iwf|$};

\draw (azero) edge[->] (addI);
\draw (addI) edge[->] (covI);
\draw (addI) edge[->] (nonI);
\draw (covI) edge[->] (sizX);
\draw (nonI) edge[->] (sizX);
\draw (covI) edge[->] (cofI);
\draw (nonI) edge[->] (cofI);
\draw (sizX) edge[->] (sizI);
\draw (cofI) edge[->] (sizI);
}
\end{tikzpicture}
\caption{Diagram of the cardinal characteristics associated with $\Iwf$. An arrow  $\mathfrak x\rightarrow\mathfrak y$ means that (provably in ZFC) 
    $\mathfrak x\le\mathfrak y$.}
\label{diag:idealI}
\end{figure}
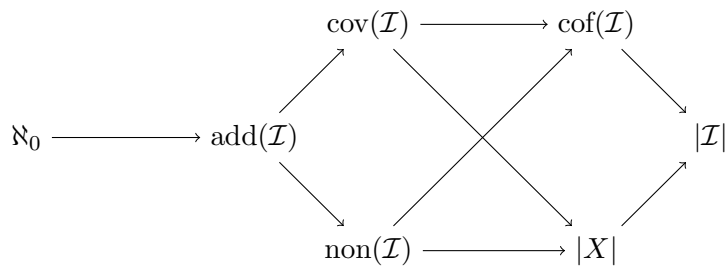

Over the years, research on the relationship between the cardinal characteristics associated with $\Ewf$  and other classical cardinal characteristics of the continuum has been conducted by many people, e.g.~\cite{Mi1981,BS1992,BJ,Brppts,Car23,CMlocalc,GaMe}. Most of them are illustrated in~\autoref{CichonplusE}, except for those that appear in~\autoref{BarJu:sr}: 

\begin{theorem}[{\cite{BS1992}, see also \cite[Thm.~2.6.9]{BJ}}]\label{ThmZFC:E}
\

\begin{enumerate}[label=\rm(\arabic*)]
    \item  $\max\{\cov(\Mwf),\cov(\Nwf)\}\leq\cov(\Ewf)\leq\max\{\dfrak,\cov(\Nwf)\}$.

    \item $\min\{\bfrak,\non(\Nwf)\}\leq\non(\Ewf)\leq\min\{\non(\Mwf),\non(\Nwf)\}$. 

    \item $\add(\Ewf)=\add(\Mwf)$ and $\cof(\Ewf)=\cof(\Mwf)$.
\end{enumerate}
\end{theorem}

See~\cite{blass} for the definitions of the following cardinals. Denote by
\begin{itemize}
    
    
    
    
    \item $\efrak$ the \emph{evasion number}.
    
    
    
    
    
    \item $\rfrak$ the \emph{reaping number}.

    \item  $\sfrak$ the \emph{splitting number}.

\end{itemize}

\begin{theorem}\label{BarJu:sr}
\ 

\begin{enumerate}[label=\rm(\arabic*)]
    \item \emph{\cite[Lem.~7.4.3]{BJ}} $\sfrak\leq\non(\Ewf)$ and $\cov(\Ewf)\leq\rfrak$. 

    \item\emph{\cite{BrendlevasionI}} $\efrak\leq\non(\Ewf)$.
\end{enumerate}    
\end{theorem}

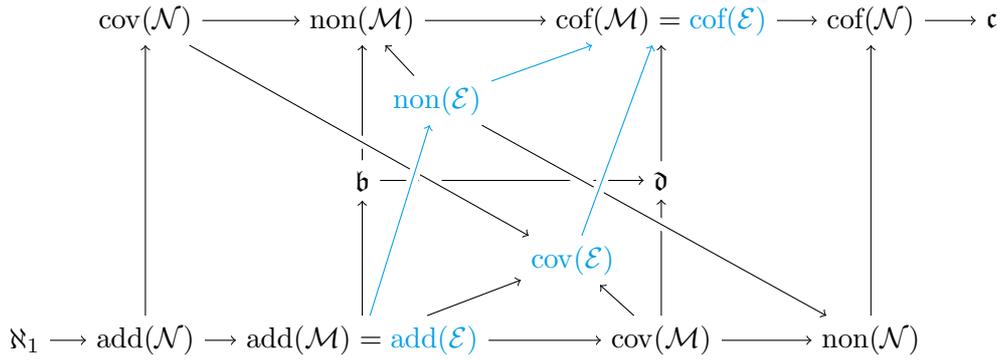
\begin{figure}[ht!]
\centering
\begin{tikzpicture}[scale=1.06]
\small{
\node (aleph1) at (-1,3) {$\aleph_1$};
\node (addn) at (0.5,3){$\add(\Nwf)$};
\node (covn) at (0.5,7){$\cov(\Nwf)$};
\node (nonn) at (9.5,3) {$\non(\Nwf)$} ;
\node (cfn) at (9.5,7) {$\cof(\Nwf)$} ;
\node (addm) at (3.19,3) {$\add(\Mwf)=\subiii{\add(\Ewf)}$} ;
\node (covm) at (6.9,3) {$\cov(\Mwf)$} ;
\node (nonm) at (3.19,7) {$\non(\Mwf)$} ;
\node (cfm) at (6.9,7) {$\cof(\Mwf)=\subiii{\cof(\Ewf)}$} ;
\node (b) at (3.19,5) {$\bfrak$};
\node (d) at (6.9,5) {$\dfrak$};
\node (c) at (11,7) {$\cfrak$};
\node (none) at (4.12,6) {\subiii{$\non(\Ewf)$}};
\node (cove) at (5.8,4) {\subiii{$\cov(\Ewf)$}};
\draw (aleph1) edge[->] (addn)
      (addn) edge[->] (covn)
      (covn) edge [->] (nonm)
      (nonm)edge [->] (cfm)
      (cfm)edge [->] (cfn)
      (cfn) edge[->] (c);

\draw
   (addn) edge [->]  (addm)
   (addm) edge [->]  (covm)
   (covm) edge [->]  (nonn)
   (nonn) edge [->]  (cfn);
\draw (addm) edge [->] (b)
      (b)  edge [->] (nonm);
\draw (covm) edge [->] (d)
      (d)  edge[->] (cfm);
\draw (b) edge [->] (d);

\draw   (none) edge [->] (nonm)
        (none) edge [sub3,->] (cfm)
        (addm) edge [->] (cove);
      
\draw (none) edge [line width=.15cm,white,-] (nonn)
      (none) edge [->] (nonn);
      
\draw (cove) edge [line width=.15cm,white,-] (covn)
      (cove) edge [<-] (covm)
      (cove) edge [<-] (covn);

\draw (addm) edge [line width=.15cm,white,-] (none)
      (addm) edge [sub3,->] (none); 

\draw (cove) edge [line width=.15cm,white,-] (cfm)
      (cove) edge [sub3,->] (cfm);
}
\end{tikzpicture}
\caption{Cicho\'n's diagram and the cardinal characteristics associated with $\Ewf$.}
\label{CichonplusE}
\end{figure}

The goal of this survey paper is to prove that the next diagram of the cardinal characteristics of $\Ewf$ can be distinguished at each position by constructing a model using the forcing technique. In this diagram, a dotted line means that we can obtain a model in which
the cardinal characteristics of the left side are strictly smaller than the cardinal characteristics of the right side. The result attached to a dotted line is the lemma in which the consistency of the inequality is proved.

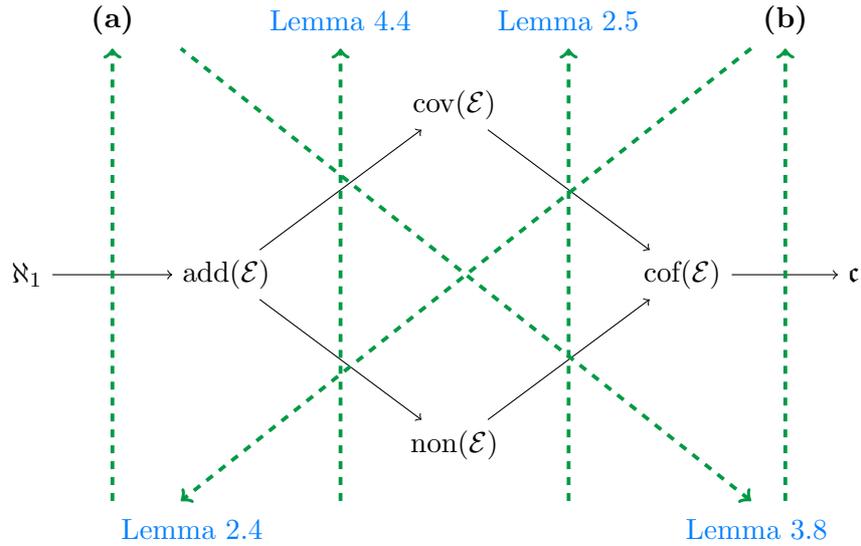
\begin{figure}[ht!]
\centering
\begin{tikzpicture}[scale=1.5]
\small{
\node (azero) at (-0.75,1) {$\aleph_1$};
\node (addI) at (1,1) {$\add(\Ewf)$};
\node (covI) at (3,2.5) {$\cov(\Ewf)$};
\node (nonI) at (3,-0.5) {$\non(\Ewf)$};
\node (cofI) at (5,1) {$\cof(\Ewf)$};
\node (sizI) at (6.5,1) {$\cfrak$};

\draw (azero) edge[->] (addI);
\draw (addI) edge[->] (covI);
\draw (addI) edge[->] (nonI);
\draw (covI) edge[->] (cofI);
\draw (nonI) edge[->] (cofI);
\draw (cofI) edge[->] (sizI);
\draw[color=sug,line width=.05cm,dashed,->] (0,-1)--(0,3);
\draw[color=sug,line width=.05cm,dashed,->] (2,-1)--(2,3);
\draw[color=sug,line width=.05cm,dashed,->] (4,-1)--(4,3);
\draw[color=sug,line width=.05cm,dashed,->] (5.9,-1)--(5.9,3);
\draw[color=sug,line width=.05cm,dashed,<-](5.6,-1)--(.6,3);
\draw[color=sug,line width=.05cm,dashed,->] (5.6,3)--(.6,-1);
\node at (0,3.25) {$\textbf{(a)}$};
\node at (5.9,3.25) {$\textbf{(b)}$};
\node at (0.7,-1.25) {\autoref{cov<cof}};
\node at (5.65,-1.25) {\autoref{non<cov}};
\node at (2,3.25) {\autoref{add<non=cov}};
\node at (4,3.25) {\autoref{cov=non<cof}};
}
\end{tikzpicture}
\end{figure}

In this diagram, there are two results that we won't be proved. These are represented by (a) and  (b) in the above diagram. (a) is the consistency of $\aleph_1<\add(\Ewf)$ and (b) is the consistency of $\cof(\Ewf)<\cfrak$. Both consistencies hold in the model $\Dor_\pi$, which is obtained by a FS iteration of length $\pi=\lambda\kappa$ of Hechler forcing where $\aleph_1\leq\kappa\leq\lambda=\lambda^{\aleph_0}$ with $\kappa$ regular (see e.g.~\cite[Thm.~5]{mejiamatrix}).  There,  $\add(\Mwf)=\cof(\Mwf)=\kappa$ and $\cfrak=\lambda$, so by using~\autoref{ThmZFC:E}, $\Dor_\pi$ forces $\add(\Ewf)=\cof(\Ewf)=\kappa$.

\noindent\textbf{Notation.} A~\emph{forcing notion} is a pair $\la\Por,\leq\ra$ where $\Por\neq\emptyset$ and $\leq$ is a relation on~$\Por$ satisfying reflexivity and 
transitivity. We also use the expression pre-ordered set (abbreviated p.o.\ set or just poset) to refer to a forcing notion. The elements of $\Por$ are called conditions and we say that a condition $q$ is \emph{stronger} than a condition $p$ if $q\leq p$.

\begin{definition}
Let $\Por$ be a forcing notion. 
\begin{enumerate}[label=\rm(\arabic*)]
    \item Say that $p, q\in\Por$ are \emph{compatible} (in $\Por$), denoted by $p\,\|_\Por\,q$, if $\exists r\in\Por\colon r\leq p\textrm{\ and\ }r\leq q$. Say that $p, q\in\Por$ are \emph{incompatible} (in $\Por$) if they are not compatible in $\Por$, which is
denoted by $p\perp_\Por q$. 

When $\Por$ is clear from the context, we just write $p\,\|\,q$ and $p\perp q$.

    \item Say that $A\subseteq\Por$ is an \emph{antichain} if $p, q\in\Por\colon p\neq q\Rightarrow p\perp q$. $A$ is a \emph{maximal antichain} on $\Por$ iff $A$ is an antichain and $\forall p\in\Por\,\exists q\in A\colon p\,\|_\Por\,q$.
    
    \item Say that $D\subseteq\Por$ is \emph{dense} (in $\Por$) if $p\in\Por\,\exists q\in D\colon q\leq p$.

    \item Say that $G\subseteq\Por$ is a \emph{$\Por$-filter} if it satisfies
    \begin{enumerate}
        \item $G\neq\emptyset$;

        \item for all $p, q\in G$ there is some $r\in G$ such that $r\leq p$ and $r\leq q$; and 

        \item if $p\in\Por$, $q\in G$ and $q\leq p$, then $p\in G$.
    \end{enumerate}

    \item Let $\mathcal D$ be a family of dense subsets of $\Por$. Say that $G\subseteq\Por$ is \emph{$\Por$-generic over $V$} if $G$ is
a $\Por$-filter and $\forall D\in\mathcal D\cap V\colon G\cap D\neq\emptyset$. Denote by $\dot G_\Por$ the canonical name of the generic set. When $\Por$ is clear from the context, we just
write  $\dot G$.
\end{enumerate}
\end{definition}

\begin{fct}[{\cite{Gold}}]\label{basfor}
  Let $\Por$ be a forcing notion. Let $p, q\in\Por$. 
   \begin{enumerate}[label=\rm(\arabic*)]
\item $p\perp q$ iff $q\Vdash p\notin\dot G$.\smallskip

\item $G\subseteq\Por$ is a $\Por$-generic over $V$ iff for every maximal antichain $A\in V$, $|G\cap A|=1$.
   \end{enumerate}
\end{fct}


Let $I$ be a set. Denote by $\Cor_I$ be the poset that adds Cohen reals indexed by $I$.

Here, as usual, given a formula $\phi$, $\forall^\infty\, n<\omega\colon \phi$ means that all but finitely many natural numbers satisfy $\phi$; $\exists^\infty\, n<\omega\colon \phi$ means that infinitely many natural numbers satisfy $\phi$

\section{The consistency of  \texorpdfstring{$\cov(\Ewf)<\non(\Ewf)$}{} and \texorpdfstring{$\non(\Ewf)=\cov(\Ewf)<\cof(\Ewf)$}{}}

Throughout this section, assume that CH holds in the ground model $V.$

One of the fundamental properties of forcing is the Laver property:

\begin{quote}
  A forcing notion $\Por\in V$ has the \textit{Laver property} if for any $\Por$-generic $G$ over $V$, any function $f\in\omega^\omega\cap V$ and any $\Por$-name $\dot g$ for a member in $\omega^\omega$ such that $\Vdash \dot g\leq^* f$, there exists a function $\varphi\in([\omega]^{<\omega})^\omega\cap V$ such that $\Vdash \dot g(n)\in\varphi(n)$ and $|\varphi(n)|\leq n+1$ for every $n\in\omega$.  
\end{quote} 

\begin{example}\label{examp}
\ 

\begin{enumerate}[label=\rm(\arabic*)]
     \item\label{examp:a} Mathias forcing (see e.g.~\cite[Sec.~7.4A]{BJ}). 

    \item\label{examp:b} Miller forcing (see e.g.~\cite[Sec.~7.3E]{BJ}).

    \item Laver forcing (see e.g.~\cite[Sec.~7.3D]{BJ}).
\end{enumerate}    
\end{example}

We now show that any poset with the Laver property preserves covering families of the ideal $\Ewf$. The following is a combinatorial consequence of the Laver property.

\begin{lemma}
Assume that $\Por$ has the Laver property. Let $\dot x$ be a name for a real in $2^\omega$, and $p\in\Por$. Then there are $q\leq p$ and a closed null set $C\subseteq2^\omega$ such that $q\Vdash\dot x\in C$.
\end{lemma}
\begin{proof}
Let $\la I_n:n\in\omega\ra$ be
a partition of $\omega$ into finite intervals such that $\min(I_0)=0$, $\max(I_n)+1=
\min(I_{n+1})$, and $|I_n|=2n+1$. Define a name $\dot h$ for a function with domain $\omega$ such that the trivial condition forces $\dot h(n)=\dot x{\upharpoonright}I_n$. In particular, $\dot h(n)$ is forced to be an element of $2^{I_n}$. By the Laver property, there are $q\leq p$ in $\Por$ and $\varphi\in\prod_{n<\omega}\Pwf(2^{I_n})$ such that $|\varphi(n)|\leq 2^n$ and $q\Vdash \dot h(n)\in\varphi(n)$ for all $n<\omega$. Hence $q\Vdash \dot x{\upharpoonright}I_n\in\varphi(n)$ for all $n$, so let $C:=\{y\in2^\omega:\forall n<\omega\colon y{\upharpoonright}I_n\in\varphi(n)\}$ (coded in the ground model). It is not hard to prove that $C$ is a closed set, so it remains to see that $C$ has measure zero set. Indeed, 
\begin{align*}
  \Lb(C)&=\prod_{n\in\omega}\Lb\big(\{y\in2^\omega:y{\upharpoonright}I_n\in\varphi(n)\}\big)\\ 
  &=\prod_{n\in\omega}\frac{|\varphi(n)|}{2^{|I_n|}}\\
  &=\prod_{n\in\omega}\frac{2^n}{2^{|I_n|}}\leq\prod_{n\in\omega}\frac{2^n}{2^{2n+1}}=\prod_{n\in\omega}\frac{1}{2^{n+1}}=0.
\end{align*}
By $\Lb$ we denote Lebesgue measure zero on $\cantor$. Since $q\Vdash \dot x\in C$, we are done.
\end{proof}

Employing the earlier result:

\begin{corollary}\label{LaverPrecov}
Assume that $\Por$ has the Laver property. Then $\Por$ preserves covering families of the ideal $\Ewf$.
\end{corollary}

As a direct consequence,  we obtain:

\begin{lemma}\label{cov<cof}
Let $\Mor_{\omega_2}$ be a CS (countable support) iteration of length $\omega_2$ of Mathias forcing. Then in $V^{\Mor_{\omega_2}}$, $\add(\Ewf)=\cov(\Ewf)=\aleph_1$, and  $\non(\Ewf)=\cof(\Ewf)=\aleph_2$. 
\end{lemma}
\begin{proof}
It is well-known that $\Mor_{\omega_2}$ forces $\cov(\Nwf)=\cov(\Mwf)=\aleph_1$ and 
$\sfrak=\bfrak=\dfrak=\non(\Ewf)=\aleph_2=\cfrak$. (see e.g.~\cite[Sec. 7.4A]{BJ}). In addition to this, since $\sfrak\leq\non(\Ewf)$ by~\autoref{BarJu:sr}, we have that $\Mor_{\omega_2}$ forces $\non(\Ewf)=\aleph_2$ and $\Mor_{\omega_2}$ forces $\cov(\Ewf)=\aleph_1$ thanks to~\autoref{examp}~\ref{examp:a}~and~\autoref{LaverPrecov}. 
\end{proof}

\begin{lemma}\label{cov=non<cof}
Let $\Mior_{\omega_2}$ be a CS iteration of length $\omega_2$ of Miller forcing. Then, in $V^{\Mior_{\omega_2}}$, $\add(\Ewf)=\cov(\Ewf)=\non(\Ewf)=\aleph_1$, and  $\cof(\Ewf)=\aleph_2$. 
\end{lemma}
\begin{proof}
It is well-known that $\Mior_{\omega_2}$ forces $\bfrak=\non(\Mwf)=\non(\Nwf)=\aleph_1$ and $\dfrak=\aleph_2$ (see e.g.~\cite[Sec. 7.3E]{BJ}). Additionally, since $\bfrak=\non(\Mwf)=\aleph_1$, by using~\autoref{ThmZFC:E}, we get that $\non(\Ewf)=\min\{\non(\Mwf),\non(\Nwf)\}$, so $\Mior_{\omega_2}$ forces $\non(\Ewf)=\aleph_1$. On the other hand, by using~\autoref{examp}~\ref{examp:b}~and~\autoref{LaverPrecov}, $\Mior_{\omega_2}$ forces $\cov(\Ewf)=\aleph_1$. Lastly, since $\cof(\Ewf)=\cof(\Mwf)$ by~\autoref{ThmZFC:E}, $\Mior_{\omega_2}$ forces $\cof(\Ewf)=\aleph_2$. 
\end{proof}

In comparison to~\autoref{cov<cof}, the upcoming result is stronger, since it can be forced $\non(\Ewf)>\dfrak$.

\begin{lemma}
Let $\Por$ be a CS iteration of length $\omega_2$ of the tree forcing from~\cite[Lem.~2]{Brppts}. Then, in $V^{\Por}$,  $\dfrak=\cov(\Ewf)=\aleph_1$ and $\non(\Ewf)=\aleph_2$.
\end{lemma}
\begin{proof}
We just are going to give only a brief outline of the proof. Details can be found in the references. The forcing from~\cite[Lem.~2]{Brppts} we iterate belongs to a class of forcing notions introduced by Shelah~\cite{Sh:326} (see also~\cite[Sec.~7.3B]{BJ}). This forcing is $\baire$-bounding and does not add random reals. Furthermore, Brendle proved that this forcing increases $\non(\Ewf)$ (see~\cite[Lem.~2]{Brppts}). We therefore have that, in $V^\Por$, $\cov(\Nwf)=\dfrak=\aleph_1$ and $\non(\Ewf)=\aleph_2$. Since $\cov(\Mwf)=\dfrak=\aleph_1$, by applying~\autoref{ThmZFC:E} $\cov(\Ewf)=\aleph_1$.
\end{proof}

\section{The consistency of  \texorpdfstring{$\non(\Ewf)<\cov(\Ewf)$}{}}

Before delving into the specifics, we provide all of the necessary to prove the consistency of $\non(\Ewf)<\cov(\Ewf)$. This section is based on~\cite[Sec.~7]{CMlocalc} and~\cite[Sec.~5]{Car23}.

\begin{definition}\label{Setslalom}
 For an increasing function $b\in\omega^\omega$, define the following sets of slaloms:
 \begin{enumerate}[label=\normalfont(\arabic*)]

     \item $\Swf_{b}:=\set{\varphi\in\prod_{n<\omega}\Pwf(b(n))}{\forall\, n\in\omega\colon \frac{|\varphi(n)|}{2^{b(n+1)-b(n)}} \leq \frac{1}{2^n}}$.

     \item $\Swfw_{b}:=\set{\varphi\in\prod_{n<\omega}\Pwf(b(n))}{\exists^\infty\, n\in\omega\colon \frac{|\varphi(n)|}{2^{b(n+1)-b(n)}} \leq \frac{1}{2^n}}$.
 \end{enumerate}
\end{definition}

Note that $\Swf_b\subseteq\Swfw_b$ and that
\[\begin{split}
    \varphi\in\Swfw_b & \text{ implies }\liminf_{n\to\infty}\frac{|\varphi(n)|}{|b(n)|}<1,\\
    \varphi\in\Swf_b & \text{ implies }\limsup_{n\to\infty}\frac{|\varphi(n)|}{|b(n)|}<1.
\end{split}\]
Also observe that, if $b$ is a function with the domain $\omega$ such that $b(i)\neq\vacio$ for all $i<\omega$, and  $h\in\omega^\omega$, then 

\begin{align*}\label{remarkE}
\liminf_{n\to\infty}\frac{h(n)}{|b(n)|}<1 & \text{\ iff\ }\forall\, m<\omega\colon \prod_{n\geq m}\frac{h(n)}{|b(n)|}=0,\tag{$\varotimes$}\\
 \limsup_{n\to\infty}\frac{h(n)}{|b(n)|}<1 & \text{\ iff\ }\forall\, A\in[\omega]^{\aleph_0}\colon \prod_{n\in A}\frac{h(n)}{|b(n)|}=0.\notag 
\end{align*}

We can construct sets in $\Ewf$ using slaloms in the following way.

\begin{lemma}\label{toolE}
Let $\tilb\in\omega^\omega$ be increasing and $b(n):=2^{\tilb(n+1)-\tilb(n)}$ for all $n\in\omega$. If $\varphi\in\prod_{n<\omega}\pts(b(n))$ and $\liminf_{n\to\infty}\frac{|\varphi(n)|}{|b(n)|}<1$, then the set
\[H_{\tilb,\varphi}:=\set{x\in2^\omega}{\forall^\infty\, n\in\omega\colon x{\upharpoonright}[\tilb(n), \tilb(n+1))\in\varphi(n)}\]
belongs to $\Ewf$.
\end{lemma}
\begin{proof}
Notice that $H_{\tilb,\varphi}$ is a countable union of the closed null sets  
\[B_{\tilb, \varphi}^m:=\set{x\in2^\omega}{\forall\, n\geq m\colon x{\upharpoonright}[\tilb(n), \tilb(n+1))\in\varphi(n)}\text{ for }m\in\omega.\] 
Indeed,  
\begin{align*}
\Lb(B_{\tilb, \varphi}^m)=&\prod_{n\geq m}\Lb\left(\set{x\in2^\omega}{x{\upharpoonright}[\tilb(n), \tilb(n+1))\in\varphi(n)}\right)\\
=&\prod_{n\geq m}\frac{|\varphi(n)|}{2^{\tilb(n+1)-\tilb(n)}} =0,
\end{align*}
where the latter equality holds by~\eqref{remarkE}. Hence $\Lb(B_{\tilb, \varphi}^m)=0$ for all $m<\omega$, so   
\[H_{\tilb, \varphi}=\bigcup_{m<\omega} B_{\tilb, \varphi}^m\] is an $F_\sigma$ null set and thus belongs to $\Ewf$.
\end{proof}

\begin{corollary}\label{cor:toolE}
    Let $b\in\omega^\omega$ be increasing. If $\varphi\in \Swfw_b$ then $H_{b,\varphi}\in\Ewf$.
\end{corollary}

Thanks to the foregoing lemma, we obtain a basis of $\Ewf$.

\begin{lemma}[{\cite[Thm. 4.3]{BS1992}}]\label{lem:combE}
Suppose that $C\in\Ewf$. Then there is some increasing $\tilb \in\omega^\omega$ and some $\varphi\in\Swf_{\tilb}$ such that $C\subseteq H_{\tilb, \varphi}$.
\end{lemma}
\begin{proof}
Let us assume wlog that $C\subseteq2^\omega$ is a null set of type $F_\sigma$. Then $C$ can be written as $\bigcup_{n\in\omega}C_n$ where $\la C_n:\, n\in\omega\ra$ is an increasing family of closed sets of measure zero. 

Note that each $C_n$ is a compact set. It is easy to see that, if $K\subseteq2^{\omega}$ is a compact null set, then 
\[\forall\, \varepsilon>0\ \forall^{\infty}\, n\ \exists\, T\subseteq2^{n}\colon K\subseteq[T] \text{ and }\frac{|T|}{2^{n}}<\varepsilon\]
where $[T]:=\bigcup_{t\in T}[t]$.
Hence, we can define an increasing $\tilb\in\omega^\omega$ by $\tilb(0):=0$ and 
\[\tilb(n+1):=\min\set{m>\tilb(n)}{\exists\, T_n\subseteq 2^{m}\colon C_n\subseteq[T_n]\text{ and }\frac{|T_n|}{2^m} < \frac{1}{4^{\tilb(n)}}}\textrm{\ for $n>1$}.\]
Next, choose $T_n\subseteq2^{\tilb(n+1)}$ such that $C_n\subseteq[T_n]$  and $\frac{|T_n|}{2^{\tilb(n+1)}} < \frac{1}{4^{\tilb(n)}}$. Now define, for $n\in\omega$, \[\varphi(n):=\set{s{\upharpoonright}[\tilb(n), \tilb(n+1))}{s\in T_n}.\]
It is clear that $|\varphi(n)|\leq |T_n|$ for every $n<\omega$, hence
\begin{equation*}\label{eqE}
  \frac{|\varphi(n)|}{2^{\tilb(n+1)-\tilb(n)}}<\frac{1}{2^{\tilb(n)}}\leq\frac{1}{2^n}. 
\end{equation*}
Thus, $\varphi\in\Swf_{\tilb}$ and, by~\autoref{cor:toolE}, $H_{\tilb, \varphi}\in\Ewf$. We also have $C\subseteq H_{\tilb, \varphi}$. 
\end{proof}

The following is a technical lemma that connects the structure of $\omega^\omega$ with $\Ewf$. Before stating it, for each increasing $f\in\omega^\omega$ define the increasing function $f^*\colon \omega\to\omega$ such that $f^*(0)=0$ and $f^*(n+1)=f(f^*(n)+1)$ for $n>0$.

\begin{lemma}[{\cite[Lem.~7.5]{CMlocalc}}]\label{lem:mon}
Let $b_0,b_1\in\omega^\omega$ be increasing functions, let $b_1^*\in\omega^\omega$ be as above, and let $\varphi\in\Swf_{b_0}$.
\begin{enumerate}[label=\normalfont(\arabic*)]
    \item\label{lem:monb} If $b_1\not\leq^*b_0$ then there is some  $\varphi^*\in\Swfw_{b_1^*}$ such that $H_{b_0, \varphi}\subseteq H_{b_1^*, \varphi^*}$.
    
    \item\label{lem:mona} If $b_0\leq^*b_1$ then there is some $\varphi_*\in\Swf_{b_1^*}$ such that $H_{b_0, \varphi}\subseteq H_{b_1^*, \varphi_*}$.
\end{enumerate}
\end{lemma}

As a result of the previous, we infer:

\begin{lemma}[{\cite[Lem.~5.1]{BS1992}}]\label{Thm:tecBS}
Suppose that $\Por$ is a ccc forcing notion. Let $\dot C$ be a $\Por$-name for a member of $\Ewf$.
 \begin{enumerate}[label=\normalfont(\arabic*)]
     \item If $\Por$ does not add dominating reals, then there exists $b\in\omega^\omega\cap V$ and a $\Por$-name $\dot\varphi$ such that $\dot\varphi\in\Swfw_b$ and $\Vdash``\dot C\subseteq H_{b, \dot \varphi}"$.

     \item If $\Por$ is $\omega^\omega$-bounding, then there exists $b\in\omega^\omega\cap V$ and a $\Por$-name $\dot\varphi$ such that $\dot\varphi\in\Swf_b$ and $\Vdash``\dot C\subseteq H_{b, \dot \varphi}"$.
 \end{enumerate}
\end{lemma}

The previous lemma is essential for proving the consistency of $\non(\Ewf)<\cov(\Ewf)$.

\begin{lemma}[{\cite[Thm.~ 5.3]{BS1992}, see also~\cite[Thm.~5.2]{Car23}}]\label{thm:pnon(E)}
Assume that~$\Bor$ is a complete Boolean algebra with a strictly positive $\sigma$-additive measure $\mu$ and let $\dot C$ be a $\Por$-name of a closed measure zero subset of $2^\omega$. Then there is a closed measure zero subset $C^*$ of $2^\omega$ (in the ground model) such that\ 
\[\forall\, x\in 2^\omega\smallsetminus C^*\colon \Vdash x\notin \dot C.\]
\end{lemma}
\begin{proof}
By employing~\autoref{Thm:tecBS}, choose $b\in\omega^\omega$ in $V$ and a $\Bor$-name $\dot\varphi$ of a member of $\Swf_b$ such that $\Vdash_\Bor\dot C\subseteq H_{b,\dot \varphi}$. 
For $s\in 2^{[b(n), b(n+1))}$ set $B_{n,s}:=\llbracket s\in\dot\varphi(n)\rrbracket\in\Bor$. 

Now, for $n<\omega$, define $\psi(n)$ by 
\[\psi(n):=\bigg\{s\in 2^{[b(n), b(n+1))}:\mu(B_{n,s})\geq2^{-\lfloor\frac{n}{2}\rfloor}\bigg\}.\]
We claim that 
\[\frac{|\psi(n)|}{2^{b(n+1)-b(n)}}\leq2^{-\lfloor\frac{n}{2}\rfloor}.\]  
Suppose that for some $n_0\in\omega$,  \[\frac{|\psi(n_0)|}{2^{b(n_0+1)-b(n_0)}}>2^{-\lfloor\frac{n_0}{2}\rfloor}.\] 
For $S\subseteq\psi(n)$ set 
\[i(S):=\max\left\{|X|:X\subseteq S,\,\mu\bigg(\bigwedge_{s\in X}B_{n,s}\bigg)>0\right\}.\]
By~\cite[Prop. 1]{Kelley}, 
\[2^{-\lfloor\frac{n}{2}\rfloor}\leq\inf\{\mu(B_{n,s}):s\in\psi(n)\}\leq\inf\bigg\{\frac{i(S)}{|S|}:\emptyset\subsetneq S\subseteq\psi(n)\bigg\},\]
in particular, \[2^{b(n_0+1)-b(n_0)-n_0}<\frac{|\psi(n_0)|}{2^{\lfloor\frac{n_0}{2}\rfloor}}\leq i(\psi(n_0)).\]
Choose $X\subseteq\psi(n_0)$ such that $|X|>2^{b(n_0+1)-b(n_0)-n_0}$ and $\mu\big(\bigwedge_{s\in X}B_{n_0,s}\big)>0$. Hence, $\bigwedge_{s\in X}B_{n_0,s}\Vdash X\subseteq\dot\varphi(n_0)$, so $\bigwedge_{s\in X}B_{n_0,s}\Vdash |\dot\varphi(n_0)|>2^{b(n_0+1)-b(n_0)-n_0}$, which is a contradiction.

Thus
\[C^*:=\{x\in2^\omega:\forall^\infty n\in\omega\colon x{\upharpoonright}[b(n), b(n+1))\in\psi(n)\}\]
is a member of $\Ewf$ by~\autoref{toolE}. 

To end the proof, let us argue that if $x\not\in C^*$, then $\Vdash_\Bor x\not\in \dot C$. Suppose that $x\not\in C^*$. Towards a contradiction assume that $p\Vdash_\Bor x\in \dot C$ for some $p\in\Bor$. Since $\Vdash_\Bor\dot C\subseteq H_{b,\dot \varphi}$, we can assume wlog that there is some $m\in\omega$ such that $p\Vdash\forall n\geq m\colon x{\upharpoonright}[b(n), b(n+1))\in\dot \varphi(n))$ and $\mu(p)>2^{-m}$.

On the other hand, since $x\not\in C^*$, we can find an $n\geq 2m$ such that $x{\upharpoonright}[b(n), b(n+1))\not\in\psi(n)$. In particular, 
\[\mu(B_{n,\,x{\upharpoonright}[b(n), b(n+1)})<\frac{1}{2^m}.\]
Now define $q:=p\smallsetminus B_{n,\,x{\upharpoonright}[b(n), b(n+1))}$. Then $\mu(q)>0$ and $q\Vdash x{\upharpoonright}[b(n), b(n+1))\not\in\dot \varphi(n)$, which is a contradiction.
\end{proof}

We are ready to prove the consistency of $\non(\Ewf)<\cov(\Ewf)$.

\begin{lemma}\label{non<cov}
Assume $\aleph_1\leq\nu\leq\lambda=\lambda^{\aleph_0}$ with $\nu$ regular. Let $\Bor_{\pi}$ be a FS iteration of random forcing of length $\pi=\lambda\nu$. Then, in $V^{\Bor_{\pi}}$,  $\non(\Ewf)=\bfrak=\aleph_1$, $\cov(\Nwf)=\non(\Mwf)=\cov(\Mwf)=\non(\Nwf)=\nu$, and $\cov(\Ewf)=\dfrak=\lambda.$
\end{lemma}
\begin{proof}
$\Bor_{\pi}$ forces $\bfrak=\aleph_1$, $\cov(\Nwf)=\non(\Mwf)=\cov(\Mwf)=\non(\Nwf)=\nu$, and $\dfrak=\lambda$ is a well-known fact (see e.g~\cite[Thm.~5.4]{Car23}). Furthermore, by employing~\autoref{thm:pnon(E)}, we can ensure that $\Bor_{\pi}$ forces $\non(\Ewf)=\aleph_1$ and $\cov(\Ewf)=\lambda$.
\end{proof}

As an immediate consequence of the foregoing, we obtain as well: 

\begin{corollary}[{\cite[Thm.~ 5.5 and~5.6]{BS1992}}]
It is consistent with ZFC that $\non(\Ewf)<\min\{\non(\Nwf), \non(\Mwf)\}$ and $\cov(\Ewf)>\max\{\cov(\Nwf),\cov(\Mwf)\}$.
\end{corollary}

We close this section by displaying another model where the consistency of $\non(\Ewf)<\cov(\Ewf)$ holds.

\begin{theorem}[{see e.g~\cite{Brendlecurse}}]
Let $\lambda$ be an infinite cardinal such that $\lambda^{\aleph_0}=\lambda$. Then $\Cbb_\lambda$ forces $\non(\Mwf)=\aleph_1$ and $\cov(\Mwf)=\cfrak=\lambda$ (\autoref{fig:cichoncohen}). In particular, $\Cbb_\lambda$ forces $\non(\Ewf)=\aleph_1$ and $\cov(\Ewf)=\lambda$.
\begin{figure}[H]
\centering
\begin{tikzpicture}
\small{
\node (aleph1) at (-1,3) {$\aleph_1$};
\node (addn) at (1,3){$\add(\Nwf)$};
\node (covn) at (1,7){$\cov(\Nwf)$};
\node (nonn) at (9,3) {$\non(\Nwf)$} ;
\node (cfn) at (9,7) {$\cof(\Nwf)$} ;
\node (addm) at (3.66,3) {$\add(\Mwf)$} ;
\node (covm) at (6.33,3) {$\cov(\Mwf)$} ;
\node (nonm) at (3.66,7) {$\non(\Mwf)$} ;
\node (cfm) at (6.33,7) {$\cof(\Mwf)$} ;
\node (b) at (3.66,5) {$\bfrak$};
\node (d) at (6.33,5) {$\dfrak$};
\node (c) at (11,7) {$\cfrak$};
\draw (aleph1) edge[->] (addn)
      (addn) edge[->] (covn)
      (covn) edge [->] (nonm)
      (nonm)edge [->] (cfm)
      (cfm)edge [->] (cfn)
      (cfn) edge[->] (c);
\draw
   (addn) edge [->]  (addm)
   (addm) edge [->]  (covm)
   (covm) edge [->]  (nonn)
   (nonn) edge [->]  (cfn);
\draw (addm) edge [->] (b)
      (b)  edge [->] (nonm);
\draw (covm) edge [->] (d)
      (d)  edge[->] (cfm);
\draw (b) edge [->] (d);

\draw[color=blue] (5,2.5)--(5,7.5); 
\draw[circle, fill=yellow,color=yellow] (2.33,5) circle (0.5);
\draw[circle, fill=yellow,color=yellow] (7.66,5) circle (0.5);
\node at (2.33,5) {$\aleph_1$};
\node at (7.66,5) {$\lambda$};
}
\end{tikzpicture}
\caption{The constellation of Cicho\'n's diagram after adding $\lambda=\lambda^{\aleph_0}$ many Cohen reals.}\label{fig:cichoncohen}
\end{figure}
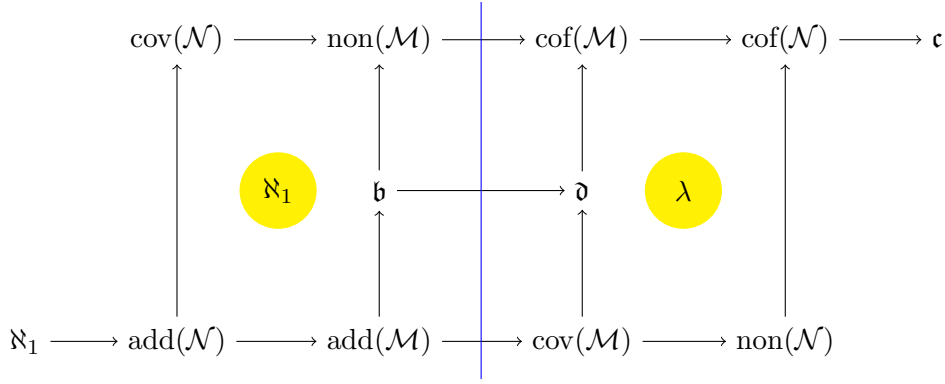
\end{theorem}

\section{The consistency of  \texorpdfstring{$\add(\Ewf)<\non(\Ewf)$}{}}

In this section, we shall use a FS iteration of $(b,h)$-localization forcing (denoted by $\LOCor_{b,h}$) proposed by Brendle and Mej\'ia~\cite{BrM} to prove the consistency of $\add(\Ewf)<\non(\Ewf)=\cov(\Ewf).$ We first start by introducing $\LOCor_{b,h}$ that we want to iterate.

\begin{definition}\label{defloc}
Given a sequence of non-empty sets $b = \seq{b(n)}{n\in\omega}$ and $h\colon \omega\to\omega$, define 
\begin{align*}
 \prod b &:= \prod_{n\in\omega}b(n),  \\
 \Swf(b,h) &:= \prod_{n\in\omega} [b(n)]^{\leq h(n)}.
\end{align*}
For two functions $x\in\prod b$ and $\varphi\in\Swf(b,h)$, we define
\[x\,\in^*\varphi\textrm{\ iff\ }\forall^\infty n\in\omega\colon x(n)\in \varphi(n).\]
We define the \emph{localization cardinals}
\begin{align*}
    \blc_{b,h}&=\min\set{|F|}{F\subseteq \prod b\;\&\;\neg\exists \varphi \in \Swf(b,h)\,\forall x \in F\colon x\in^* \varphi},\\
     \dlc_{b,h}&=\min\set{|G|}{G\subseteq \Swf(b,h)\;\&\;\forall x \in \prod b\,\exists \varphi \in G\colon x\in^* \varphi}.
\end{align*}
\end{definition}

Using the notion of~\autoref{defloc}, we define the following poset meant to increase $\non(\Ewf)$.

\begin{definition}
Fix $b$ and $h$ as in \autoref{defloc}. Define the localization forcing $\LOCor_{b,h}$ by \[\LOCor_{b,h}:=\set{(p,n)}{p\in\Swf(b,h),\, n<\omega\textrm{\ and\ }\exists m<\omega\,\forall i<\omega\colon |p(i)|\leq m},\]
ordered by $(p',n')\leq(p,n)$ iff $n\leq n'$, $p'{\upharpoonright}n=p$, and $\forall i<\omega\colon p(i)\subseteq p'(i)$. 

Let $G\subseteq\LOCor_{b,h}$ be a $\LOCor_{b,h}$-generic filter over $V$. In $V[G]$, define the generic slalom by
\[\dot\varphi_{\gen}(i)=\bigcup\{p(i):\, (p,n)\in G\},\] and for any $x\in\prod b\cap V$, there is some $n\in\omega$ such that for any $m\geq n$, $x(m)\in\dot\varphi_\gen(m)$. Consequently, $\LOCor_{b,h}$ increases $\blc_{b,h}$.

$\LOCor_{b,h}$ is $\sigma$-$m$-linked for all $2\leq m<\omega$ when $h$ diverges to infinity and $b(n)$ is countable for all $n<\omega$.
\end{definition}

Why does $\LOCor_{b,h}$ increase $\non(\Ewf)$? Because there exists a connection between $\non(\Ewf)$ and $\blc_{b,h}$. Indeed, we have:

\begin{lemma}[{\cite[Fact.~7.7 and Lem.~7.8]{CMlocalc}}]\label{uppbE}
If $b\in\omega^\omega$ and $\limsup_{n\to\infty}\frac{h(n)}{b(n)}<1$, then  $\cov(\Ewf)\leq\dlc_{b,h}$ and $\blc_{b,h}\leq\non(\Ewf)$.
\end{lemma}

We now prove the consistency of $\add(\Ewf)<\non(\Ewf)=\cov(\Ewf).$

\begin{lemma}\label{add<non=cov}
Assume $\aleph_1\leq\kappa\leq\lambda=\lambda^{\aleph_0}$ with $\kappa$ regular. Let $\Por_{\pi}$ be a FS iteration of the (b, h)-localization forcing of length $\pi=\lambda\kappa$. Then, in $V^{\Por_{\pi}}$, $\bfrak=\aleph_1$, $\non(\Ewf)=\cov(\Ewf)=\kappa$, and $\cfrak=\lambda$. 
\end{lemma}
\begin{proof}
In $V$, find increasing functions $ h, b$ in $\omega^\omega$ such that
$\limsup_{n\to\infty}\frac{h(n)}{b(n)}<1$, so by~\autoref{uppbE} $\cov(\Ewf)\leq\dlc_{b,h}$ and $\blc_{b,h}\leq\non(\Ewf)$.  In view of~\cite[Lem.~2.9]{CM22}, $\Por_\pi$ forces  $\non(\Ewf)=\blc_{b,h}=\cov(\Ewf)=\dlc_{b,h}=\kappa$. Moreover, $\Por_\pi$ forces $\cfrak=\lambda$. It remains to prove that $\Por_\pi$ forces $\bfrak=\aleph_1$. For this, it suffices to show that $\LOCor_{b,h}$ is $\sigma$-Fr-linked. 
\end{proof}

Before attempting to prove that $\LOCor_{b,h}$ is $\sigma$-Fr-linked, we begin with some notation: 

\begin{itemize}
    \item Denote by $\Fr:=\set{\omega\menos a}{a\in[\omega]^{<\aleph_0}}$ the \emph{Fr\'echet filter}.

    \item A filter $F$ on $\omega$ is \emph{free} if $\Fr\subseteq F$. A set $x\subseteq\omega$ is \emph{$F$-positive} if it intersects every member of $F$. Denote by $F^+$ the family of all $F$-positive sets. Note that $x\in\Fr^+$ iff $x$ is an infinite subset of $\omega$.
\end{itemize}

For not adding dominating reals, we have the following notion.

\begin{definition}[{\cite{mejiavert, BCM}}]\label{Def:Fr}
    Let $\Por$ be a poset and $F$ be a filter on $\omega$. A set $Q\subseteq \Por$ is \emph{$F$-linked} if, for any $\bar p=\seq{p_n}{n<\omega} \in Q^\omega$, there is some $q\in \Por$ forcing that \[q\Vdash\set{n\in\omega}{p_n\in\dot G}\in F^+.\]
    Observe that, in the case $F=\Fr$, the above equation is “$\set{n\in\omega}{p_n\in\dot G}$ is infinite”. 
    
    We say that $Q$ is \emph{uf-linked (ultrafilter-linked)} if it is $F$-linked for any filter $F$ on $\omega$ containing the \emph{Fr\'echet filter} $\Fr$.
    
    For an infinite cardinal $\mu$, $\Por$ is \emph{$\mu$-$F$-linked} if $\Por = \bigcup_{\alpha<\mu}Q_\alpha$ for some $F$-linked $Q_\alpha$ ($\alpha<\mu$). When these $Q_\alpha$ are uf-linked, we say that $\Por$ is \emph{$\mu$-uf-linked}.
\end{definition}

For ccc posets we have:

\begin{lemma}[{\cite[Lem~5.5]{mejiavert}}]\label{quasiuf}
If $\Por$ is ccc then any subset of $\Por$ is uf-linked iff it is Fr-linked. 
\end{lemma}

Our objective is to demonstrate that $\LOCor_{b,h}$ is $\sigma$-$\Fr$-linked, witnessed by
\[L_{b,h}(s,m):=\set{(p,n)\in\LOCor_{b, h}}{s\subseteq p,\, n=|s|, \textrm{\ and\ }\forall i<\omega\colon |p(i)|\leq m}\]
for $s\in\Scal_{<\omega}(b,h)=\bigcup_{k\in\omega}\prod_{n<k}[b(n)]^{\le h(n)}$, and $m<\omega$.

In order to see that $L_{b,h}(s,m)$ is Fr-linked, it suffices to prove:

\begin{lemma}[{\cite{mejiavert}, see also~\cite{Car23}}]\label{mejiavertLOC}
Let $b,h\in\omega^\omega$ be such that $\forall i<\omega\colon b(i)>0$ and $h$ goes to infinity, let $D$ be a non-principal ultrafilter on $\omega$ and $(s,m)\in\Scal_{<\omega}(b,h)\times \omega$. If $\bar p=\la (p_n, |s|):n<\omega\ra$ is a sequence in $L_{b,h}(s,m)$ then there is a $(q,n)\in L_{b,h}(s,m)$ such that if $a\in D$, then $(q,n)$ forces that $\set{n\in\omega}{p_n\in\dot G}\cap a\neq\emptyset$ is infinite.
\end{lemma}
\begin{proof}
Assume that $\bar p=\la (p_n,|s|):n<\omega\ra$ is a sequence in $L_{b,h}(s,m)$. For each $i\geq|s|$, since $[b(i)]^{\leq m}$ is finite, we can find $q(i) \in [b(i)]^{\leq m}$ and an $a_i\in D$
such that $p_n(i)=q(i)$ for any $n\in a_i$. To get $(q,n)\in L_{b,h}(s, m)$ we define  $n:=|s|$ and $q(i):= s(i)$ for all $i<n$. It remains to show that $(q,n)$ forces $\set{n\in\omega}{p_n\in\dot G}\cap a\neq\emptyset$ whenever $a\in D$. To see this, assume that $(r,n^*)\leq (q,n)$ and
$a\in D$. Since $(r,n^*)\in\LOCor_{b,h}$, we can find $k, m_0 < \omega$ such that $k\geq n^*\geq n=|s|$, $|r(i)| \leq m_0$ for any $i<\omega$, and for any $i \geq k$, $m_0 + m \leq h(i)$. Choose some $n_0\in\bigcap_{i<k} a_i\cap a$ (put $a_i:=\omega$ for $i < |s|$). Note that $p_{n_0}(i)=q(i)$ for all $i<k$.

Now we define $q'(i)$ by 
\[q'(i)=\left\{\begin{array}{ll}
                r(i)  & \text{if $i < k$,}\\
                 r(i)\cup p_{n_0}(i) & \text{if $i\geq k$.}
                 \end{array}\right.\]
Since $|q'(i)| \leq m_0 + m \leq h(i)$ for all $i\geq k$, $(q',k)$ is a condition in $\LOCor_{b,h}$. Moreover, $(q', k)$ is a condition stronger than $r$ and $p_{n_0}$, so it forces $n_0\in \set{n\in\omega}{p_n\in\dot G}\cap a$.
\end{proof}

The next step now is to indicate that Fr-linked behaves well regarding preserving unbounded families on $\baire$, which helps to keep $\bfrak$ small in generic extension. This will be proved in a sequence of lemmas.

\begin{lemma}[The author and Mej\'ia]\label{bassoft}
Let $\Por$ be a poset and $Q\subseteq\Por$. Then $Q$ is Fr-linked iff, for each $\Por$-name $\dot n$ of a natural number there is some $m<\omega$ such that $\forall p\in Q\,(p\not\Vdash m<\dot n)$.   
\end{lemma}
\begin{proof}
The direction from left into right is due to Mej\'ia~\cite[Lem~3.26]{mejiavert}. Assume that, for any $n\in\omega$, there is some $p_n \in Q$ such that $p_n\Vdash m\leq\dot n$
Then, if $G$ is $\Por$-generic over $V$, then $\set{n\in\omega}{p_n\in\dot G}$ must be finite because $p_n\in\dot G \Rightarrow n \leq \dot m[G]<\omega$. Therefore, in $V$, $Q$ cannot be Fr-linked.

On the other hand, assume that $Q$ is not Fr-linked, accordingly pick some sequence $\seq{p_{n}}{n < \omega}$ in $Q$ such that $\Vdash``\set{n\in\omega}{p_n\in \dot G}$ finite", so there is a $\Por$-name $\dot n$ of a natural number such that $\Vdash\set{n\in\omega}{p_n\in \dot G}\subseteq\dot n$. Towards a contradiction suppose that $\exists m\,\forall q\in Q\colon q\not\Vdash m<\dot n$. Choose $m\in\omega$ such that $\forall q\in Q\colon q\not\Vdash m<\dot n$. For any $n\in\omega$, $p_n\not\Vdash m<\dot n $. Next, find 
$r\leq p_m$ such that $r\Vdash m\geq\dot n$, on the other hand, since  $r\Vdash p_m\in G$, $r\Vdash m<\dot n$. This is a contradiction.
\end{proof}

\begin{lemma}[{\cite{Mejiacurse}}]\label{bassoftII}
Let $\Por$ be a poset and $Q$ be an $\Fr$-linked subset of $\Por$.  If $\dot y$ is a $\Por$-name of a member of $\baire$, then there is $y'\in\baire$ (in the ground model) such that, for any $x\in\baire$ \[x\not\leq^* y'\Rightarrow\forall n\in\omega\,\forall p\in Q\colon p\not\Vdash\forall m\geq n\colon x(m)\leq\dot y(m).\]
\end{lemma}
\begin{proof}
By applying~\autoref{bassoft}, for each $m\in\omega$ find $y'\in\baire$ such that, for each $m<\omega$ no member of $Q$ forces $y'(m)<\dot y(m)$. 

Now suppose that $x\in\baire$ and $x\not\leq^* y'$. Towards a contradiction  assume that there is a $n\in\omega$ and $p\in Q$ so that \[p\Vdash\forall m\geq n\colon x(m)\leq\dot y(m).\]
Choose $m\geq n$ so that $x(m)>y'(m)$. On the other hand, $p\not\Vdash y'(m)<\dot y(m)$, so there is some $q\leq p$ such that $q\Vdash y'(m)\geq\dot y(m)$. Then 
\[q\Vdash \dot y(m)\geq x(m)>y'(m)\geq\dot y(m),\]
which is a contradiction.
\end{proof}
 
Using the earlier results about Fr-linkedness, we infer that $\Por_\pi$ forces $\bfrak=\aleph_1$. So we are done with the proof of~\autoref{add<non=cov}.

We now provide another proof for the consistency of $\add(\Ewf)<\non(\Ewf)=\cov(\Ewf).$ 

\begin{theorem}[The author and Mej\'ia]\label{thm:Eventit}
Assume $\aleph_1\leq\kappa\leq\lambda=\lambda^{\aleph_0}$ with $\kappa$ regular. Let $\Eor_{\pi}$ be a FS iteration of eventually different real forcing $\Eor$ of length $\pi=\lambda\kappa$. Then, in $V^{\Eor_{\pi}}$, $\cov(\Nwf)=\bfrak=\aleph_1$, $\non(\Ewf)=\cov(\Ewf)=\kappa$ and $\non(\Nwf)=\dfrak=\cfrak=\lambda$  (see \autoref{fig:cichonEventdiff}).
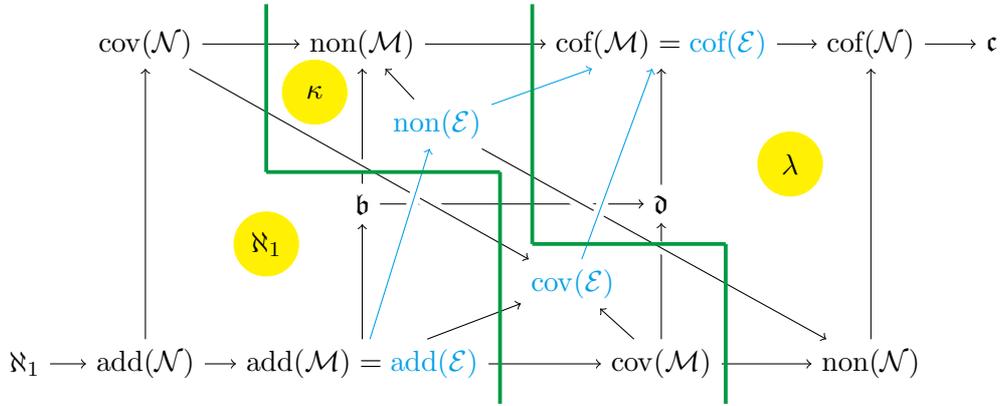
\begin{figure}[ht]
\centering
\begin{tikzpicture}[scale=1.06]
\small{
\node (aleph1) at (-1,3) {$\aleph_1$};
\node (addn) at (0.5,3){$\add(\Nwf)$};
\node (covn) at (0.5,7){$\cov(\Nwf)$};
\node (nonn) at (9.5,3) {$\non(\Nwf)$} ;
\node (cfn) at (9.5,7) {$\cof(\Nwf)$} ;
\node (addm) at (3.19,3) {$\add(\Mwf)=\subiii{\add(\Ewf)}$} ;
\node (covm) at (6.9,3) {$\cov(\Mwf)$} ;
\node (nonm) at (3.19,7) {$\non(\Mwf)$} ;
\node (cfm) at (6.9,7) {$\cof(\Mwf)=\subiii{\cof(\Ewf)}$} ;
\node (b) at (3.19,5) {$\bfrak$};
\node (d) at (6.9,5) {$\dfrak$};
\node (c) at (11,7) {$\cfrak$};
\node (none) at (4.12,6) {\subiii{$\non(\Ewf)$}};
\node (cove) at (5.8,4) {\subiii{$\cov(\Ewf)$}};
\draw (aleph1) edge[->] (addn)
      (addn) edge[->] (covn)
      (covn) edge [->] (nonm)
      (nonm)edge [->] (cfm)
      (cfm)edge [->] (cfn)
      (cfn) edge[->] (c);

\draw
   (addn) edge [->]  (addm)
   (addm) edge [->]  (covm)
   (covm) edge [->]  (nonn)
   (nonn) edge [->]  (cfn);
\draw (addm) edge [->] (b)
      (b)  edge [->] (nonm);
\draw (covm) edge [->] (d)
      (d)  edge[->] (cfm);
\draw (b) edge [->] (d);

\draw   (none) edge [->] (nonm)
        (none) edge [sub3,->] (cfm)
        (addm) edge [->] (cove);
      
\draw (none) edge [line width=.15cm,white,-] (nonn)
      (none) edge [->] (nonn);
      
\draw (cove) edge [line width=.15cm,white,-] (covn)
      (cove) edge [<-] (covm)
      (cove) edge [<-] (covn);

\draw (addm) edge [line width=.15cm,white,-] (none)
      (addm) edge [sub3,->] (none); 

\draw (cove) edge [line width=.15cm,white,-] (cfm)
      (cove) edge [sub3,->] (cfm);
\draw[color=sug,line width=.05cm] (2,7.5)--(2,5.4); 
\draw[color=sug,line width=.05cm] (2,5.4)--(4.9,5.4);
\draw[color=sug,line width=.05cm] (4.9,5.4)--(4.9,2.5);
\draw[color=sug,line width=.05cm] (5.3,7.5)--(5.3,4.5); 
\draw[color=sug,line width=.05cm] (5.3,4.5)--(7.7,4.5); 
\draw[color=sug,line width=.05cm] (7.7,4.5)--(7.7,2.5);

\draw[circle, fill=yellow,color=yellow] (2,4.5) circle (0.4);
\draw[circle, fill=yellow,color=yellow] (2.6,6.4) circle (0.4);
\draw[circle, fill=yellow,color=yellow] (8.5,5.5) circle (0.4);
\node at (2,4.5) {$\aleph_1$};
\node at (2.6,6.4) {$\kappa$};
\node at (8.5,5.5) {$\lambda$};
}
\end{tikzpicture}
\caption{Cichon's diagram after adding $\kappa$-many eventually different reals with $\Eor$.}\label{fig:cichonEventdiff}
\end{figure}
\end{theorem}
\begin{proof}
It is well-known that $\Eor_{\pi}$ forces $\cov(\Nwf)=\bfrak=\aleph_1$, $\non(\Mwf)=\cov(\Mwf)=\kappa$ and $\non(\Nwf)=\dfrak=\cfrak=\lambda$. (see e.g.~\cite[Sec.~3.1]{CM22}). Moreover, we know that $\Eor$ is Fr-linked by~\cite[Lem.~3.29]{mejiavert}. It remains to prove that $\Eor_\pi$ forces $\non(\Ewf)=\cov(\Ewf)=\kappa$. For this, it suffices to see that $\Eor$ forces all ground model reals belongs to $\Ewf$.

Think of $\Eor$ as \[\Eor=\largeset{(s,\varphi)}{s\in\omega^{<\omega}\textrm{\ and\ }\varphi\in\bigcup_{m\in\omega}\Swf(\omega,m)}\]
Ordered by $(s',\varphi')\leq(s,\varphi)$ iff $s\subseteq s'$ and $\varphi(n)\subseteq\varphi'(n)$ for all $n\in\omega$, and $s'(n)\not\in\varphi(n)$ for any $n\in|s'|\smallsetminus|s|$. 
Let $G$ be a $\Eor$-generic filter over $V$. In $V[G]$, define
\[e_\gen=\bigcup\set{s}{\exists\varphi\colon (s,\varphi)\in G}.\]
Notice that for any $x\in\baire\cap V$, $\forall^\infty n\in\omega\colon x(n)\neq e_\gen(n)$.
\begin{clm}\label{fctE}
$\Vdash_\Eor\baire\cap V\subseteq\set{x\in\baire}{\forall^\infty n\in\omega\colon x(n)\neq\dot e_\gen(n)}\in\Ewf$.    
\end{clm}
\begin{proof}
    It is clear that $\set{x\in\baire}{\forall^\infty n\in\omega\colon x(n)\neq\dot e_\gen(n)}=\bigcup_{m\in\omega}\bigcap_{n\geq m}\set{x\in\baire}{ x(n)\neq\dot e_\gen(n)}$ is an $F_\sigma$ set and $\Vdash_\Eor\baire\cap V\subseteq\set{x\in\baire}{\forall^\infty n\in\omega\colon x(n)\neq\dot e_\gen(n)}$, so it remains to prove that in $V[e]$, $\Lb_\omega(\set{x\in\baire}{\forall^\infty n\in\omega\colon x(n)\neq e_\gen(n)})=0$ (by $\Lb_\omega$ we denote Lebesgue measure zero on $\baire$). Indeed, note that 

$\Lb_\omega(\set{x\in\baire}{\forall^\infty n\in\omega\colon x(n)\neq e_\gen(n)})=\lim_{n\to\infty}\prod_{m\geq n}\bigg(1-\frac{1}{2^{e_\gen(m)+1}}\bigg)$ and 
\[
 \prod_{m\geq n}\bigg(1-\frac{1}{2^{e_\gen(m)+1}}\bigg)\leq\prod_{m\geq n}\mathbf{e}^{-\frac{1}{2^{e_\gen(m)+1}}}
 \leq \mathbf{e}^{-\Sigma_{m\geq n}\frac{1}{2^{e_\gen(m)+1}}}.
\]
To conclude the proof, let us argue 
\begin{equation}\label{clmE}
\Vdash_\Eor\Sigma_{n\in\omega}\frac{1}{2^{\dot e_\gen(n)+1}}=\infty.
\tag{\faSubway}
\end{equation}
For this, it suffices to see that \[\forall n\,\forall p\in\Eor\,\exists q\leq p\,\exists m\geq n\,\exists m'>m\colon q\Vdash\Sigma_{i=m}^{m'+1}\frac{1}{2^{\dot e_\gen(i)+1}}>1.\]
Let $n$ and $p=(s,\varphi)\in\Eor$, with $M:=\sup\set{|\varphi(i)|}{i<\omega}$. Wlog assume that $|s|\geq n$. Extend $p$ to $q=(t,\varphi)$  such that $t\supseteq s$ and $\forall i\in |t|\smallsetminus|s|\colon t(i)\leq M$ and $|t|-|s|>2^{M+1}$. Put $m:=|s|$ and $m':=|t|$. Then it can be proved that $q$ is as desired.

Utilizing~(\ref{clmE}), we conclude that in $V[e]$, $\Lb_\omega(\set{x\in\baire}{\forall^\infty n\in\omega\colon x(n)\neq e_\gen(n)})=0$.
\end{proof}
By~employing~\autoref{fctE}, it can be proved that $\Eor_\pi$ forces $\non(\Ewf)=\cov(\Ewf)=\kappa$.
\end{proof}

We will conclude this section by presenting an additional concept related to not adding dominating reals, which indeed matches Fr-linked.

Brendle and Judah~\cite{BrJ} presented the property that helps us not add dominating reals. Given a partial order $\Por$, a function $h\colon\Por\to\omega$ is a \emph{height function} iff $q\leq p$ implies $h(q)\geq h(p)$. A pair $\la\Por,h\ra$ fulfills the property~(\ref{BrJprop}) iff:
\begin{equation}
 \parbox{0.8\textwidth}{
 given a maximal antichain $\set{p_n}{n\in\omega}\subseteq\Por$ and $m\in\omega$, there is an $n\in\omega$ such that: whenever $p$ is incompatible with $\set{p_j}{j\in n}$ then $h(p)>m$.  
}
\tag{\faBasketballBall}
\label{BrJprop}
\end{equation}
They proved that FS iterations of ccc posets with the property~(\ref{BrJprop}) do not add dominating reals. Inspired by~(\ref{BrJprop}), in~\cite{Carsoft}, the author introduced the following linkedness property:

\begin{definition}[{\cite[Def.~3.1]{Carsoft}}]\label{maindef}
  Let $\Por$ be a poset. A set $Q\subseteq\Por$ is \emph{$\mathrm{leaf}$-linked} if, for any maximal antichain $\set{p_n}{n\in\omega}\subseteq\Por$, there is some $n\in\omega$ such that $\forall p \in Q\,\exists j<n\colon p\parallel p_j.$

  For an infinite cardinal $\theta$. Say that $\Por$ is \emph{$\theta$-$\mathrm{leaf}$-linked} if $\Por = \bigcup_{\alpha<\theta}Q_\alpha$ where $Q_\alpha$ is $\mathrm{leaf}$-linked ($\alpha<\mu$). 
\end{definition}

Indeed, Fr-linked and leaf-linked coincide for ccc posets. 

\begin{lemma}[{\cite[Lem.~3.6]{Carsoft}}]\label{thm:a1}
Let $\Por$ be a ccc poset and a set $Q\subseteq\Por$. The following statements are equivalent:
\begin{enumerate}[label=\rm (\arabic*)]
    \item\label{Fr} $Q$ is $\Fr$-linked.
    
    \item\label{Soft} $Q$ is $\mathrm{leaf}$-linked.

    \item\label{ConFr} for each $\Por$-name $\dot n$ of a natural number there is some $m<\omega$ such that $\forall p\in Q\colon p\not\Vdash m<\dot n$. 
\end{enumerate}
\end{lemma}
\begin{proof}
$\ref{Fr}\Rightarrow\ref{Soft}$. Let $\set{p_n}{n\in\omega}$ be a maximal antichain in $\Por$. Towards a contradiction assume that $\forall m\,\exists q_m\in Q\,\forall n<m\colon q_m\perp p_n$. By $\Fr$-linkedness, choose $q\in\Por$ such that $q\Vdash``\exists^\infty m\in\omega\colon q_m\in\dot G"$. Since each $q_m$ fulfills $\forall n<m\colon q_m\perp p_n$, $q\Vdash\exists^\infty m\,\forall n<m\colon p_n\not\in\dot G$. Hence, $q\Vdash\forall n<\omega\colon p_n\not\in\dot G$. On the other hand, since $\set{p_n}{n\in\omega}$ is a maximal antichain, by using~\autoref{basfor}, $\Vdash \exists n\in\omega\colon p_n\in\dot G$, which reaches a contradiction.

$\ref{Soft}\Rightarrow\ref{ConFr}$. Assume that $\dot n$ is a $\Por$-name of a natural number and assume towards a contradiction that for each $m\in\omega$, there is some $q_m\in Q$ that forces $m<\dot n$. Let $\set{p_k}{k\in\omega}$ be a maximal antichain deciding the value $\dot n$ and let $n_k$ be such that $p_k\Vdash``\dot n=n_k"$. By~soft-linkedness, choose $m\in\omega$ such that $\forall p\in Q\,\exists k<m\colon p\parallel p_k$. Next, let $k^\star$ be larger than $\max\set{n_k}{k<m}$. Then $p_{k^\star}\Vdash k^\star<\dot n$ and since $q^\star\in Q$, there is $k<m$ such that $p_{k^\star}\parallel p_k$. Let  $r\leq q_{k^\star}, p_k$. Then $r\Vdash \dot n=k<k^\star<\dot n$, which is a contradiction.

$\ref{ConFr}\Rightarrow\ref{Fr}$. This implication follows by~\autoref{bassoft}.
\end{proof}

As an immediate consequence, we infer:

\begin{corollary}
If $\Por$ is ccc then any subset of $\Por$ is leaf-linked iff it is Fr-linked.    
\end{corollary}

Next, we establish some consequences of our leaf-linked property:

\begin{lemma}[{\cite[Lem.~3.9]{Carsoft}}]\label{solft:one}
Let $\Por$ be a poset and $\theta$ be a cardinal number. Consider the following properties of $\Pbb$:
\begin{enumerate}[label=\rm (\arabic*)]
\item \label{solft:onea}
$\Pbb$ is $\theta$-leaf-linked.
\item\label{solft:oneb}
There is a~dense set $Q\subseteq\Pbb$ and a 
function $h\colon Q\to\theta$ such that, for every $\alpha<\theta$, the set\/ $\set{p\in Q}{ h(p)=\alpha}$ is leaf-linked.

\item \label{solft:onec} There is a~dense set $Q\subseteq\Pbb$ and a
function $h\colon Q\to\theta$ such that, for every $\alpha<\theta$, the set\/ $\set{p\in Q}{ h(p)\le\alpha}$ is leaf-linked.
\end{enumerate}
Then $\ref{solft:onec}\Rightarrow\ref{solft:oneb}$; $\ref{solft:onea}\Rightarrow\ref{solft:oneb}$; $\ref{solft:onec}\Rightarrow\ref{solft:onea}$; and
$\ref{solft:onea}\Leftrightarrow\ref{solft:oneb}\Leftrightarrow\ref{solft:onec}$ for $\theta=\omega$.
\end{lemma}
\begin{proof}
 The implication \indent$\ref{solft:onec}\Rightarrow\ref{solft:oneb}$ is trivial.

\noindent$\ref{solft:onea}\Rightarrow\ref{solft:oneb}$. Let $\langle Q_\alpha\mid\alpha<\theta\rangle$ be a~sequence of leaf-linked
subsets of~$\Pbb$ witnessing that $\Pbb$~is $\theta$-leaf-linked.
The set $Q:=\bigcup_{\alpha<\theta}Q_\alpha$ is a~dense subset of~$\Pbb$.
For $p\in Q$ define
$h(p)=\min\set{\alpha<\theta}{ p\in Q_\alpha}$. It is clear that every set $\set{p\in Q}{ h(p)=\alpha}$ is a~subset of the leaf-linked set~$Q_\alpha$ and is therefore leaf-linked.

\noindent$\ref{solft:onec}\Rightarrow\ref{solft:onea}$ for regular~$\theta$.
For every $\alpha<\theta$ the set $Q_\alpha=\set{p\in Q}{h(p)\le\alpha}$
is leaf-linked and $Q=\bigcup_{\alpha<\theta}Q_\alpha$ is a~dense subset
of~$\Pbb$.

\noindent$\ref{solft:oneb}\Rightarrow\ref{solft:onec}$ for $\theta=\omega$.
Note that $\set{p\in Q}{h(p)\le n}$ is a~finite union of leaf-linked sets
$\set{p\in Q}{h(p)=k}$ for $k\le n$ and is therefore leaf-linked.   
\end{proof}

We now give a useful sufficient condition under which the pairs $\la\Pbb,h\ra$ are $\sigma$-leaf-linked. We begin by considering the following property of pairs $\la\Pbb,h\ra$ where $\Por$ is a poset and $h$ is a height function on $\Por$: Say that $\la\Pbb,h\ra$ has property~$(\clubsuit)$ if:
\begin{equation}
 \parbox{0.8\textwidth}{
\begin{enumerate}[label=\rm(\Roman*)]
   \item if $\set{p_n}{n<\omega}$ is decreasing and $\exists m\in\omega\,\forall n\in\omega\,(h(p_n)\leq m)$, then $\exists p\in\Por\,\forall n\in\omega\,(p\leq p_n)$;

     \item $\forall m\in\omega\,\forall P\in[\Por]^{<\omega}\smallsetminus\{\emptyset\}\,\exists R\in[\Por]^{<\omega}\,(R\perp P)$ and $\forall p\in\Por\,(h(p)\leq m)\,[p\perp P\Rightarrow\exists r\in R\,(p\le r)]$;

   
\end{enumerate}
 }
 \label{Brtrevol}
 \tag{$\clubsuit$}
 \end{equation}

The following lemma justifies the introduction of the property~\eqref{Brtrevol}:

\begin{lemma}[{\cite[Lem.~1.2]{BrJ}, see also~\cite[Lem.~3.13]{Carsoft}}]\label{Br:lemma}
A~ccc poset $\Por$ with a~height function $h\colon\Por\to\omega$
satisfying~\emph{(\ref{Brtrevol})} is $\sigma$-leaf-linked.
\end{lemma}
\begin{proof}
For $m\in\omega$ let \[Q_m:=\set{p\in\Por}{h(p)\leq m}\textrm{\ and\ } Q:=\bigcup_{m\in\omega}Q_m.\] 
We show that each $Q_m$ is leaf-linked, so assume that $\set{p_n}{n\in\omega}\subseteq\Por$ is a  maximal antichain in $\Por$ and show that there is some $n\in\omega$ such that $\forall q \in Q_m\,\exists j<n\,(q\,||\,p_j)$. Suppose not. So 
for each $n\in\omega$ let $R_n$ be a finite subset of $\Por$ obtained by employing~(II) of~(\ref{Brtrevol}) to the family $P=\set{p_i}{i<n}$. For each $n\in\omega$, enumerate $R_n$ as $\set{q_j^n}{j<k_n}$ for $k_n<\omega$. 

By assumption, none of these sets can be empty and we can assume that $q_j^n\in Q_m$ for all $j, n$. By~applying~(II) of~(\ref{Brtrevol}) they form
an $\omega$-tree with finite levels concerning $``\leq"$. Then by K\"{o}nig's lemma, there is a function $f\in\omega^\omega$ such that $q^0_{f(0)}\leq q^1_{f(1)}\leq q^3_{f(3)}\leq\cdots.$ By~(I) of~(\ref{Brtrevol}) there is a condition $q\leq q_{f(n)}^n$ for all $n$, contradicting
the fact that $\set{p_n}{n\in\omega}$ is a maximal antichain.   

Finally, $\Por$ is $\sigma$-leaf-linked by~\ref{solft:onec} of~\autoref{solft:one} applied to $\theta=\omega$ and for the height function $h$.
\end{proof}

We also consider the following property of pairs $\la\Pbb,h\ra$ where $\Por$ is a poset and $h$ is a height function on $\Por$: Say that $\la\Pbb,h\ra$ has property (\ref{Miropr}) if:
\begin{equation}
 \parbox{0.8\textwidth}{
 given $m\in\omega$ and a 
 finite non-maximal non-empty antichain $P\subseteq\Por$, there is some finite subset $R$ of $\Por$ such that $R\perp P$ and $\forall p\in\Pbb$
$[(p\perp P$ and $h(p)\le m)\Rightarrow\exists r\in R\, (p\le r)]$.
}
\tag{\faBug}
\label{Miropr}
\end{equation}

\begin{lemma}[{\cite[Lem.~3.11]{Carsoft}}]
A~ccc poset $\Por$ with a~height function $h\colon\Por\to\omega$
satisfying~\emph{(\ref{Miropr})} is $\sigma$-leaf-linked.
\end{lemma}
\begin{proof}
For $m\in\omega$ let \[Q_m:=\set{p\in\Por}{h(p)\leq m}\textrm{\ and\ } Q:=\bigcup_{m\in\omega}Q_m.\] 
We prove that every set $Q_m$ is leaf-linked. Let $P=\set{p_k}{k\in\omega}$ be arbitrary maximal antichain in~$\Pbb$.
Denote $P_n=\set{p_k}{k\le n}$.
By~using~(\ref{Miropr}), for every set~$P_n$ there is some finite subset $R_n$ of $\Por$ such that
$R_n\perp P_n$ and $\forall p\in Q_m\,[p\perp P_n\Rightarrow\exists r\in R_n\,(p\le r)$.
We can assume that $R_n\in[Q_m]^{<\omega}$ because
$(p\in Q_m$ and $p\le r)\Rightarrow r\in Q_m$.
Therefore,
\begin{equation*}
\forall n\in\omega\, \exists R_n\in[Q_m]^{<\omega}\smallsetminus\,\{\varnothing\}\,
[R_n\perp P_n\text{ and }
\forall p\in Q_m\, (p\perp P_n\Rightarrow\exists r\in R_n\, (r\le p))].
\tag{$\boxplus$}
\label{ch:soft}
\end{equation*}
Employing~(\ref{ch:soft}) for every $p\in R_{n+1}$ there is $r\in R_n$ such that $r\le p$
(because $p\in Q_m$ and $p\perp P_n$).
In this way, by induction, we prove that
$\forall n\in\omega\,\forall p\in R_n\,\exists r\in R_0)\,(r\le p)$.
Denote by $R_n'$ the set $\set{r\in R_0}{\exists p\in R_n\,(r\le p)}$.
Then~(\ref{ch:soft}) holds with the sets $R_n'$ instead of~$R_n$
and therefore without loss of generality, we can assume that $R_n\subseteq R_0$
for all $n\in\omega$.
Since $P$~is a~maximal antichain, there is $m_0\in\omega$
such that $\forall p\in R_0\,(p\not\perp P_{m_0})$. Then it follows that $R_{m_0}=\varnothing$ because $R_{m_0}\perp P_{m_0}$ and
$R_{m_0}\subseteq R_0$.
Then by~(\ref{ch:soft}) for $n=m_0$ we get $\forall p\in Q_m\,(p\not\perp P_{m_0})$.
This finishes the proof that $Q_m$~is leaf-linked. Then, $\Por$ is $\sigma$-leaf-linked by~\ref{solft:onec} of~\autoref{solft:one} applied to $\theta=\omega$ and for the height function $h$.
\end{proof}

\begin{example}[{\cite{BrJ}}]
\

\begin{enumerate}[label=\rm(\arabic*)]
    \item Consider random forcing $\Bor$ as
\[\Bor=\set{B\subseteq\cantor}{\forall s\in2^{<\omega}\colon[t]\cap B\neq\emptyset\Rightarrow\Lb([t]\cap B)>0}.\]
We define $h\colon\Bor\to\omega$ by $h(B)=\min\set{n\in\omega}{\Lb(B)\geq\frac{1}{n}}$. Brendle and Judah established that $\la\Bor, h\ra$ satisfies the property~\eqref{Brtrevol}. Hence, by using~\autoref{Br:lemma}, $\Bor$ is $\sigma$-leaf-linked.
    \item For $m\in\omega$ put \[B\cap2^m:=\set{t\in2^m}{[t]\cap B\neq\emptyset}.\]
We define the poset $\Bor^*$ as $\set{(B,n)}{B\in\Bor\textrm{\ and\ }n\in\omega}$. This forcing is ccc and it is used to add a perfect set of random reals. Now, define $\Bor_0^*$ by
\[\Bor_0^*=\set{(B,n)}{\forall s\in 2^n\cap B\colon\Lb([s]\cap B)\geq2^{\dom(s)+1}}.\]
It is clear that $\Bor_0^*\subseteq\Bor^*$ and $\Bor_0^*$
 is dense in $\Bor^*$ (recall the Lebesgue density). We define a height on $\Bor^*_0$ as follows: For $(B,n)$ let $h_1((B,n))=n$. In~\cite[Lem.~1.5]{BrJ}, it was demonstrated that $\la \Bor^*_0, h_1\ra$ fulfills the property~\eqref{Brtrevol}, so by~using~\autoref{Br:lemma} again, we obtain that $\Bor^*_0$ is $\sigma$-leaf-linked. Since $\Bor_0^*$
 is dense in $\Bor^*$, $\Bor_0$ is still $\sigma$-leaf-linked. 
\end{enumerate}
\end{example}

\begin{example}[{\cite{BrendlevasionI}}]\label{Brpr}
 Given a $b\in\omega^\omega$ define the forcing $\Pror_b$ as the poset whose conditions are of the form $p=(A,\seq{ \pi_n}{n\in A}, F)$ such that 
\begin{enumerate}[label=(\roman*)]
    \item $A\in[\omega]^{<\aleph_0}$,
    \item $\forall n\in A\colon\pi_n\colon\prod_{m<n}b(m)\to b(n)$, and 
    \item $F$ is a finite set of functions and $\forall f\in F\,\exists k\leq\omega\colon f\in\prod_{n<k}b(n)$, i.e., $F$ is a finite subset of $\Seq_{<\omega}(b)\cup\Seq(b)$.
\end{enumerate}
The order is defined by 
$(A',\pi',F')\le(A,\pi,F)$ iff
\begin{enumerate}[label=(\alph*)]

\item\label{Brpr:a}
$A\subseteq A'$, $\pi\subseteq\pi'$,
\item\label{Brpr:b}
if $A'\ne A$ and $A\ne\emptyset$, then $\max A<\min(A'\smallsetminus A)$,
\item\label{Brpr:c}
$\forall f\in F\,\exists g\in F'\colon f\subseteq g$, and
\item\label{Brpr:d}
$\forall f\in F\,\forall n\in (A'\smallsetminus A)\cap\dom(f)\colon\pi'_n(f{\restriction}n)=f(n)$.
\end{enumerate}

Furthermore, $\Pror_b$ is $\sigma$-centered (and thus ccc). Let $h\colon\Pror_b\to\omega$ be defined by $h(p)=\max(A_p\cup\{|F_p|\})$, which can be  easily seen that is a~height function where $p=(A_p,\pi_p,F_p)$. Brendle proved that the pair $\la\Pror_b,h\ra$ satisfies the conditions of the property~(\ref{Brtrevol}). Therefoe it is $\sigma$-leaf-linked ($\sigma$-$\Fr$-linked).
\end{example}

Motivated for all the discussions in the preceding sections, one could ask the following:  

\begin{problem}\label{ProbE:tot}
Are each one of the following constellations consistent with ZFC?
\begin{enumerate}[label=\rm(\arabic*)]
    \item\label{ProbE:tot:a} The constellation on the left side in~\autoref{4Etotal}.

    \item\label{ProbE:tot:b} The constellation on the right side in~\autoref{4Etotal}.
\end{enumerate}
\end{problem}

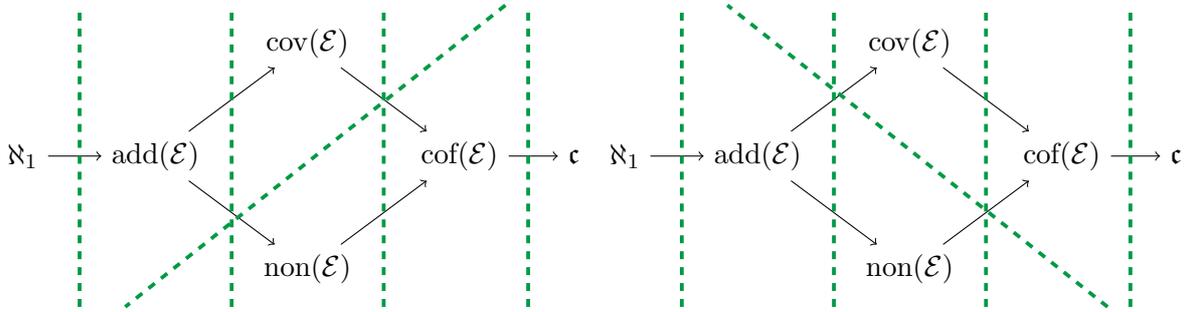
\begin{figure}[ht]
\centering
\begin{tikzpicture}[scale=1,every node/.style={scale=1}]
\small{
\node (azero) at (-0.75,1) {$\aleph_1$};
\node (addI) at (1,1) {$\add(\Ewf)$};
\node (covI) at (3,2.5) {$\cov(\Ewf)$};
\node (nonI) at (3,-0.5) {$\non(\Ewf)$};
\node (cofI) at (5,1) {$\cof(\Ewf)$};
\node (sizI) at (6.5,1) {$\cfrak$};

\draw (azero) edge[->] (addI);
\draw (addI) edge[->] (covI);
\draw (addI) edge[->] (nonI);
\draw (covI) edge[->] (cofI);
\draw (nonI) edge[->] (cofI);
\draw (cofI) edge[->] (sizI);
\draw[color=sug,line width=.05cm,dashed] (0,-1)--(0,3);
\draw[color=sug,line width=.05cm,dashed] (2,-1)--(2,3);
\draw[color=sug,line width=.05cm,dashed] (4,-1)--(4,3);
\draw[color=sug,line width=.05cm,dashed] (5.9,-1)--(5.9,3);
\draw[color=sug,line width=.05cm,dashed] (5.6,3)--(.6,-1);
}
\end{tikzpicture}
\centering
\begin{tikzpicture}[scale=1,every node/.style={scale=1}]
\small{
\node (azero) at (-0.75,1) {$\aleph_1$};
\node (addI) at (1,1) {$\add(\Ewf)$};
\node (covI) at (3,2.5) {$\cov(\Ewf)$};
\node (nonI) at (3,-0.5) {$\non(\Ewf)$};
\node (cofI) at (5,1) {$\cof(\Ewf)$};
\node (sizI) at (6.5,1) {$\cfrak$};

\draw (azero) edge[->] (addI);
\draw (addI) edge[->] (covI);
\draw (addI) edge[->] (nonI);
\draw (covI) edge[->] (cofI);
\draw (nonI) edge[->] (cofI);
\draw (cofI) edge[->] (sizI);
\draw[color=sug,line width=.05cm,dashed] (0,-1)--(0,3);
\draw[color=sug,line width=.05cm,dashed] (2,-1)--(2,3);
\draw[color=sug,line width=.05cm,dashed] (4,-1)--(4,3);
\draw[color=sug,line width=.05cm,dashed] (5.9,-1)--(5.9,3);
\draw[color=sug,line width=.05cm,dashed](5.6,-1)--(.6,3);
}
\end{tikzpicture}
\caption{Possible models where the four cardinal characteristics associated with $\Ewf$ can be pairwise different.}
\label{4Etotal}
\end{figure}

We have solved~\autoref{ProbE:tot}~\ref{ProbE:tot:b}, concretely, we proved:

\begin{theorem}[{\cite[Thm.~5.6]{Car23}}]\label{SepofE}
Assuming $\lambda_0\leq\lambda_1\leq\lambda_2\leq\lambda_3\leq\lambda_4$ are uncounatble cardinals satisfying certain conditions. Then there is a ccc poset forcing the constellation of~\autoref{Fig:SepofE}.
    
    \begin{figure}[H]
\centering   
    \begin{tikzpicture}[scale=1.06]
\small{
\node (aleph1) at (-1,3) {$\aleph_1$};
\node (addn) at (0.5,3){$\add(\Nwf)$};
\node (covn) at (0.5,7){$\cov(\Nwf)$};
\node (nonn) at (9.5,3) {$\non(\Nwf)$} ;
\node (cfn) at (9.5,7) {$\cof(\Nwf)$} ;
\node (addm) at (3.19,3) {$\add(\Mwf)=\subiii{\add(\Ewf)}$} ;
\node (covm) at (6.9,3) {$\cov(\Mwf)$} ;
\node (nonm) at (3.19,7) {$\non(\Mwf)$} ;
\node (cfm) at (6.9,7) {$\cof(\Mwf)=\subiii{\cof(\Ewf)}$} ;
\node (b) at (3.19,5) {$\bfrak$};
\node (d) at (6.9,5) {$\dfrak$};
\node (c) at (11,7) {$\cfrak$};
\node (none) at (4.12,6) {\subiii{$\non(\Ewf)$}};
\node (cove) at (5.8,4) {\subiii{$\cov(\Ewf)$}};
\draw (aleph1) edge[->] (addn)
      (addn) edge[->] (covn)
      (covn) edge [->] (nonm)
      (nonm)edge [->] (cfm)
      (cfm)edge [->] (cfn)
      (cfn) edge[->] (c);

\draw
   (addn) edge [->]  (addm)
   (addm) edge [->]  (covm)
   (covm) edge [->]  (nonn)
   (nonn) edge [->]  (cfn);
\draw (addm) edge [->] (b)
      (b)  edge [->] (nonm);
\draw (covm) edge [->] (d)
      (d)  edge[->] (cfm);
\draw (b) edge [->] (d);

\draw   (none) edge [->] (nonm)
        (none) edge [sub3,->] (cfm)
        (addm) edge [->] (cove);
      
\draw (none) edge [line width=.15cm,white,-] (nonn)
      (none) edge [->] (nonn);
      
\draw (cove) edge [line width=.15cm,white,-] (covn)
      (cove) edge [<-] (covm)
      (cove) edge [<-] (covn);

\draw (addm) edge [line width=.15cm,white,-] (none)
      (addm) edge [sub3,->] (none); 

\draw (cove) edge [line width=.15cm,white,-] (cfm)
      (cove) edge [sub3,->] (cfm);

\draw[color=sug,line width=.05cm] (-0.45,5.5)--(4.9,5.5);
\draw[color=sug,line width=.05cm] (1.45,5.5)--(1.45,2.5);
\draw[color=sug,line width=.05cm] (-0.45,2.5)--(-0.45,7.5);
\draw[color=sug,line width=.05cm] (4.9,7.5)--(4.9,2.5);
\draw[color=sug,line width=.05cm] (4.9,4.5)--(11,4.5);

\draw[circle, fill=yellow,color=yellow] (0.95,4) circle (0.4);
\draw[circle, fill=yellow,color=yellow] (2.2,4) circle (0.4);
\draw[circle, fill=yellow,color=yellow] (1.2,6) circle (0.4);
\draw[circle, fill=yellow,color=yellow] (7.4,3.6) circle (0.4);
\draw[circle, fill=yellow,color=yellow] (8.2,5.5) circle (0.4);
\node at (8.2,5.5) {$\lambda_4$};
\node at (0.95,4) {$\lambda_0$};
\node at (2.2,4) {$\lambda_1$};
\node at (1.2,6) {$\lambda_2$};
\node at (7.4,3.6) {$\lambda_3$};
}
\end{tikzpicture}
\caption{Constellation forced in~\autoref{SepofE}}
\label{Fig:SepofE}
\end{figure}
\end{theorem} 

We conclude this section by presenting several problems related to~\autoref{SepofE}, which we are interested in:

\begin{problem}\label{prob:E}
Are each one of the following constellations consistent with ZFC?
\begin{enumerate}[label=\rm(\arabic*)]
    \item\label{CMUZ2} Constellation of~\autoref{Fig:cmuz}.
    
    \item\label{probE:a} Constellation of~\autoref{FigEa}.

    \item\label{probE:b} Constellation of~\autoref{FigEb}.

    \item\label{probE:c} Constellation of~\autoref{FigEc}.

    \item\label{probE:d} Constellation of~\autoref{FigEd}.

    \item\label{probE:e} Constellation of~\autoref{FigEe}.
    
    \item\label{probE:f} Constellation of~\autoref{FigEf}.

    \item\label{probE:g} Constellation of~\autoref{FigEg}.

    \item\label{probE:h} Constellation of~\autoref{FigEh}.

     \item\label{probE:i} Constellation of~\autoref{FigEi}.
     
     \item\label{probE:j} Constellation of~\autoref{FigEj}.
\end{enumerate}
\end{problem}

\begin{figure}[ht!]
\centering
\begin{tikzpicture}[scale=1.06]
\small{
\node (aleph1) at (-1,3) {$\aleph_1$};
\node (addn) at (0.5,3){$\add(\Nwf)$};
\node (covn) at (0.5,7){$\cov(\Nwf)$};
\node (nonn) at (9.5,3) {$\non(\Nwf)$} ;
\node (cfn) at (9.5,7) {$\cof(\Nwf)$} ;
\node (addm) at (3.19,3) {$\add(\Mwf)=\subiii{\add(\Ewf)}$} ;
\node (covm) at (6.9,3) {$\cov(\Mwf)$} ;
\node (nonm) at (3.19,7) {$\non(\Mwf)$} ;
\node (cfm) at (6.9,7) {$\cof(\Mwf)=\subiii{\cof(\Ewf)}$} ;
\node (b) at (3.19,5) {$\bfrak$};
\node (d) at (6.9,5) {$\dfrak$};
\node (c) at (11,7) {$\cfrak$};
\node (none) at (4.12,6) {\subiii{$\non(\Ewf)$}};
\node (cove) at (5.8,4) {\subiii{$\cov(\Ewf)$}};
\draw (aleph1) edge[->] (addn)
      (addn) edge[->] (covn)
      (covn) edge [->] (nonm)
      (nonm)edge [->] (cfm)
      (cfm)edge [->] (cfn)
      (cfn) edge[->] (c);

\draw
   (addn) edge [->]  (addm)
   (addm) edge [->]  (covm)
   (covm) edge [->]  (nonn)
   (nonn) edge [->]  (cfn);
\draw (addm) edge [->] (b)
      (b)  edge [->] (nonm);
\draw (covm) edge [->] (d)
      (d)  edge[->] (cfm);
\draw (b) edge [->] (d);

\draw   (none) edge [->] (nonm)
        (none) edge [sub3,->] (cfm)
        (addm) edge [->] (cove);
      
\draw (none) edge [line width=.15cm,white,-] (nonn)
      (none) edge [->] (nonn);
      
\draw (cove) edge [line width=.15cm,white,-] (covn)
      (cove) edge [<-] (covm)
      (cove) edge [<-] (covn);

\draw (addm) edge [line width=.15cm,white,-] (none)
      (addm) edge [sub3,->] (none); 

\draw (cove) edge [line width=.15cm,white,-] (cfm)
      (cove) edge [sub3,->] (cfm);

\draw[color=sug,line width=.05cm] (4.95,2.5)--(4.95,8.5);
\draw[color=sug,line width=.05cm] (0.75,6.45)--(4.95,6.45); 
\draw[color=sug,line width=.05cm] (-0.55,5.45)--(4.95,5.45);
\draw[color=sug,line width=.05cm] (1.35,2.5)--(1.35,5.45);
\draw[color=sug,line width=.05cm] (-0.55,2.5)--(-0.55,8.5);
\draw[color=sug,line width=.05cm] (1.8,6.45)--(1.8,8.5);
\draw[color=sug,line width=.05cm] (0.75,5.45)--(0.75,6.45);

\draw[circle, fill=yellow,color=yellow] (1.35,5.95) circle (0.4);
\draw[circle, fill=yellow,color=yellow] (0,4) circle (0.4);
\draw[circle, fill=yellow,color=yellow] (0,5.95) circle (0.4);
\draw[circle, fill=yellow,color=yellow] (2.5,4) circle (0.4);
\draw[circle, fill=yellow,color=yellow] (8.2,5) circle (0.4);
\draw[circle, fill=yellow,color=yellow] (3.19,7.8) circle (0.4);
\node at (8.2,5) {$\lambda_5$};
\node at (1.35,5.95) {$\lambda_3$};
\node at (0,4) {$\lambda_0$};
\node at (0,5.95) {$\lambda_1$};
\node at (2.5,4) {$\lambda_2$};
\node at (3.19,7.8) {$\lambda_4$};
}
\end{tikzpicture}
\caption{Constellation forced in~\autoref{prob:E}~\ref{CMUZ2}}
\label{Fig:cmuz}
\end{figure}

\begin{figure}[ht!]
\centering
\begin{tikzpicture}[scale=1.06]
\small{
\node (aleph1) at (-1,3) {$\aleph_1$};
\node (addn) at (0.5,3){$\add(\Nwf)$};
\node (covn) at (0.5,7){$\cov(\Nwf)$};
\node (nonn) at (9.5,3) {$\non(\Nwf)$} ;
\node (cfn) at (9.5,7) {$\cof(\Nwf)$} ;
\node (addm) at (3.19,3) {$\add(\Mwf)=\subiii{\add(\Ewf)}$} ;
\node (covm) at (6.9,3) {$\cov(\Mwf)$} ;
\node (nonm) at (3.19,7) {$\non(\Mwf)$} ;
\node (cfm) at (6.9,7) {$\cof(\Mwf)=\subiii{\cof(\Ewf)}$} ;
\node (b) at (3.19,5) {$\bfrak$};
\node (d) at (6.9,5) {$\dfrak$};
\node (c) at (11,7) {$\cfrak$};
\node (none) at (4.12,6) {\subiii{$\non(\Ewf)$}};
\node (cove) at (5.8,4) {\subiii{$\cov(\Ewf)$}};
\draw (aleph1) edge[->] (addn)
      (addn) edge[->] (covn)
      (covn) edge [->] (nonm)
      (nonm)edge [->] (cfm)
      (cfm)edge [->] (cfn)
      (cfn) edge[->] (c);

\draw
   (addn) edge [->]  (addm)
   (addm) edge [->]  (covm)
   (covm) edge [->]  (nonn)
   (nonn) edge [->]  (cfn);
\draw (addm) edge [->] (b)
      (b)  edge [->] (nonm);
\draw (covm) edge [->] (d)
      (d)  edge[->] (cfm);
\draw (b) edge [->] (d);

\draw   (none) edge [->] (nonm)
        (none) edge [sub3,->] (cfm)
        (addm) edge [->] (cove);
      
\draw (none) edge [line width=.15cm,white,-] (nonn)
      (none) edge [->] (nonn);
      
\draw (cove) edge [line width=.15cm,white,-] (covn)
      (cove) edge [<-] (covm)
      (cove) edge [<-] (covn);

\draw (addm) edge [line width=.15cm,white,-] (none)
      (addm) edge [sub3,->] (none); 

\draw (cove) edge [line width=.15cm,white,-] (cfm)
      (cove) edge [sub3,->] (cfm);

\draw[color=sug,line width=.05cm] (4.95,2.5)--(4.95,8.5);
\draw[color=sug,line width=.05cm] (0.75,6.45)--(4.95,6.45); 
\draw[color=sug,line width=.05cm] (-0.55,5.45)--(4.95,5.45);
\draw[color=sug,line width=.05cm] (1.35,2.5)--(1.35,5.45);
\draw[color=sug,line width=.05cm] (-0.55,2.5)--(-0.55,8.5);
\draw[color=sug,line width=.05cm] (1.8,6.45)--(1.8,8.5);
\draw[color=sug,line width=.05cm] (0.75,5.45)--(0.75,6.45);

\draw[circle, fill=yellow,color=yellow] (1.35,5.95) circle (0.4);
\draw[circle, fill=yellow,color=yellow] (0,4) circle (0.4);
\draw[circle, fill=yellow,color=yellow] (0,5.95) circle (0.4);
\draw[circle, fill=yellow,color=yellow] (2.5,4) circle (0.4);
\draw[circle, fill=yellow,color=yellow] (8.2,5) circle (0.4);
\draw[circle, fill=yellow,color=yellow] (3.19,7.8) circle (0.4);
\node at (8.2,5) {$\lambda_5$};
\node at (1.35,5.95) {$\lambda_3$};
\node at (0,4) {$\lambda_0$};
\node at (0,5.95) {$\lambda_2$};
\node at (2.5,4) {$\lambda_1$};
\node at (3.19,7.8) {$\lambda_4$};
}
\end{tikzpicture}
\caption{Constellation of~\autoref{prob:E}~\ref{probE:a}}
\label{FigEa}
\end{figure}

\begin{figure}[ht!]
\centering
\begin{tikzpicture}[scale=1.06]
\small{
\node (aleph1) at (-1,3) {$\aleph_1$};
\node (addn) at (0.5,3){$\add(\Nwf)$};
\node (covn) at (0.5,7){$\cov(\Nwf)$};
\node (nonn) at (9.5,3) {$\non(\Nwf)$} ;
\node (cfn) at (9.5,7) {$\cof(\Nwf)$} ;
\node (addm) at (3.19,3) {$\add(\Mwf)=\subiii{\add(\Ewf)}$} ;
\node (covm) at (6.9,3) {$\cov(\Mwf)$} ;
\node (nonm) at (3.19,7) {$\non(\Mwf)$} ;
\node (cfm) at (6.9,7) {$\cof(\Mwf)=\subiii{\cof(\Ewf)}$} ;
\node (b) at (3.19,5) {$\bfrak$};
\node (d) at (6.9,5) {$\dfrak$};
\node (c) at (11,7) {$\cfrak$};
\node (none) at (4.12,6) {\subiii{$\non(\Ewf)$}};
\node (cove) at (5.8,4) {\subiii{$\cov(\Ewf)$}};
\draw (aleph1) edge[->] (addn)
      (addn) edge[->] (covn)
      (covn) edge [->] (nonm)
      (nonm)edge [->] (cfm)
      (cfm)edge [->] (cfn)
      (cfn) edge[->] (c);

\draw
   (addn) edge [->]  (addm)
   (addm) edge [->]  (covm)
   (covm) edge [->]  (nonn)
   (nonn) edge [->]  (cfn);
\draw (addm) edge [->] (b)
      (b)  edge [->] (nonm);
\draw (covm) edge [->] (d)
      (d)  edge[->] (cfm);
\draw (b) edge [->] (d);

\draw   (none) edge [->] (nonm)
        (none) edge [sub3,->] (cfm)
        (addm) edge [->] (cove);
      
\draw (none) edge [line width=.15cm,white,-] (nonn)
      (none) edge [->] (nonn);
      
\draw (cove) edge [line width=.15cm,white,-] (covn)
      (cove) edge [<-] (covm)
      (cove) edge [<-] (covn);

\draw (addm) edge [line width=.15cm,white,-] (none)
      (addm) edge [sub3,->] (none); 

\draw (cove) edge [line width=.15cm,white,-] (cfm)
      (cove) edge [sub3,->] (cfm);

\draw[color=sug,line width=.05cm] (4.95,2.5)--(4.95,8.5);
\draw[color=sug,line width=.05cm] (0.75,6.45)--(4.95,6.45); 
\draw[color=sug,line width=.05cm] (-0.55,5.45)--(4.95,5.45);
\draw[color=sug,line width=.05cm] (1.35,2.5)--(1.35,5.45);
\draw[color=sug,line width=.05cm] (-0.55,2.5)--(-0.55,8.5);
\draw[color=sug,line width=.05cm] (1.8,6.45)--(1.8,8.5);
\draw[color=sug,line width=.05cm] (0.75,5.45)--(0.75,6.45);

\draw[circle, fill=yellow,color=yellow] (1.35,5.95) circle (0.4);
\draw[circle, fill=yellow,color=yellow] (0,4) circle (0.4);
\draw[circle, fill=yellow,color=yellow] (0,5.95) circle (0.4);
\draw[circle, fill=yellow,color=yellow] (2.5,4) circle (0.4);
\draw[circle, fill=yellow,color=yellow] (8.2,5) circle (0.4);
\draw[circle, fill=yellow,color=yellow] (3.19,7.8) circle (0.4);
\node at (8.2,5) {$\lambda_5$};
\node at (1.35,5.95) {$\lambda_2$};
\node at (0,4) {$\lambda_0$};
\node at (0,5.95) {$\lambda_3$};
\node at (2.5,4) {$\lambda_1$};
\node at (3.19,7.8) {$\lambda_4$};
}
\end{tikzpicture}
\caption{Constellation of~\autoref{prob:E}~\ref{probE:b}}
\label{FigEb}
\end{figure}

\begin{figure}[ht!]
\centering
\begin{tikzpicture}[scale=1.06]
\small{
\node (aleph1) at (-1,3) {$\aleph_1$};
\node (addn) at (0.5,3){$\add(\Nwf)$};
\node (covn) at (0.5,7){$\cov(\Nwf)$};
\node (nonn) at (9.5,3) {$\non(\Nwf)$} ;
\node (cfn) at (9.5,7) {$\cof(\Nwf)$} ;
\node (addm) at (3.19,3) {$\add(\Mwf)=\subiii{\add(\Ewf)}$} ;
\node (covm) at (6.9,3) {$\cov(\Mwf)$} ;
\node (nonm) at (3.19,7) {$\non(\Mwf)$} ;
\node (cfm) at (6.9,7) {$\cof(\Mwf)=\subiii{\cof(\Ewf)}$} ;
\node (b) at (3.19,5) {$\bfrak$};
\node (d) at (6.9,5) {$\dfrak$};
\node (c) at (11,7) {$\cfrak$};
\node (e) at (2,4.8) {$\efrak$};
\node (none) at (4.12,6) {\subiii{$\non(\Ewf)$}};
\node (cove) at (5.8,4.25) {\subiii{$\cov(\Ewf)$}};
\draw (aleph1) edge[->] (addn)
      (addn) edge[->] (covn)
      (covn) edge [->] (nonm)
      (nonm)edge [->] (cfm)
      (cfm)edge [->] (cfn)
      (cfn) edge[->] (c);

\draw
   (addn) edge [->]  (addm)
   (addm) edge [->]  (covm)
   (covm) edge [->]  (nonn)
   (nonn) edge [->]  (cfn);
\draw (addm) edge [->] (b)
      (b)  edge [->] (nonm);
\draw (covm) edge [->] (d)
      (d)  edge[->] (cfm);
\draw (b) edge [->] (d);

\draw   (none) edge [->] (nonm)
        (none) edge [sub3,->] (cfm)
        (addm) edge [->] (cove);
      
\draw (none) edge [line width=.15cm,white,-] (nonn)
      (none) edge [->] (nonn);
      
\draw (cove) edge [line width=.15cm,white,-] (covn)
      (cove) edge [<-] (covm)
      (cove) edge [<-] (covn);

\draw (addm) edge [line width=.15cm,white,-] (none)
      (addm) edge [sub3,->] (none); 

\draw (cove) edge [line width=.15cm,white,-] (cfm)
      (cove) edge [sub3,->] (cfm);

\draw (e) edge [line width=.15cm,white,-] (none)
      (e) edge [->] (none);   

\draw (e) edge [line width=.15cm,white,-] (covm)
      (e) edge [->] (covm);  

\draw (e) edge [line width=.15cm,white,-] (aleph1)
      (e) edge [->] (aleph1);

\draw[color=sug,line width=.05cm] (4.95,2.5)--(4.95,8.5);
\draw[color=sug,line width=.05cm] (0.75,6.45)--(4.95,6.45); 
\draw[color=sug,line width=.05cm] (-0.55,5.45)--(4.95,5.45);
\draw[color=sug,line width=.05cm] (-0.55,2.5)--(-0.55,8.5);
\draw[color=sug,line width=.05cm] (1.8,6.45)--(1.8,8.5);
\draw[color=sug,line width=.05cm] (0.75,5.45)--(0.75,6.45); 
\draw[color=sug,line width=.05cm] (0.65,4.3)--(0.65,5.45);
\draw[color=sug,line width=.05cm] (0.65,4.3)--(2.5,4.3);
\draw[color=sug,line width=.05cm] (1.4,2.5)--(1.4,4.3);
\draw[color=sug,line width=.05cm] (2.5,4.3)--(2.5,5.45);

\draw[circle, fill=yellow,color=yellow] (1.35,5.95) circle (0.4);
\draw[circle, fill=yellow,color=yellow] (0,4.9) circle (0.4);
\draw[circle, fill=yellow,color=yellow] (0,5.95) circle (0.4);
\draw[circle, fill=yellow,color=yellow] (2.5,3.7) circle (0.4);
\draw[circle, fill=yellow,color=yellow] (8.2,5) circle (0.4);
\draw[circle, fill=yellow,color=yellow] (3.19,7.8) circle (0.4);
\draw[circle, fill=yellow,color=yellow] (1.2,4.9) circle (0.4);
\node at (8.2,5) {$\lambda_6$};
\node at (1.35,5.95) {$\lambda_4$};
\node at (1.2,4.9) {$\lambda_3$};
\node at (0,4.9) {$\lambda_0$};
\node at (0,5.95) {$\lambda_1$};
\node at (2.5,3.7) {$\lambda_2$};
\node at (3.19,7.8) {$\lambda_5$};
}
\end{tikzpicture}
\caption{Constellation of~\autoref{prob:E}~\ref{probE:c}}
\label{FigEc}
\end{figure}

\begin{figure}[ht!]
\centering
\begin{tikzpicture}[scale=1.06]
\small{
\node (aleph1) at (-1,3) {$\aleph_1$};
\node (addn) at (0.5,3){$\add(\Nwf)$};
\node (covn) at (0.5,7){$\cov(\Nwf)$};
\node (nonn) at (9.5,3) {$\non(\Nwf)$} ;
\node (cfn) at (9.5,7) {$\cof(\Nwf)$} ;
\node (addm) at (3.19,3) {$\add(\Mwf)=\subiii{\add(\Ewf)}$} ;
\node (covm) at (6.9,3) {$\cov(\Mwf)$} ;
\node (nonm) at (3.19,7) {$\non(\Mwf)$} ;
\node (cfm) at (6.9,7) {$\cof(\Mwf)=\subiii{\cof(\Ewf)}$} ;
\node (b) at (3.19,5) {$\bfrak$};
\node (d) at (6.9,5) {$\dfrak$};
\node (c) at (11,7) {$\cfrak$};
\node (e) at (2,4.8) {$\efrak$};
\node (none) at (4.12,6) {\subiii{$\non(\Ewf)$}};
\node (cove) at (5.8,4.25) {\subiii{$\cov(\Ewf)$}};
\draw (aleph1) edge[->] (addn)
      (addn) edge[->] (covn)
      (covn) edge [->] (nonm)
      (nonm)edge [->] (cfm)
      (cfm)edge [->] (cfn)
      (cfn) edge[->] (c);

\draw
   (addn) edge [->]  (addm)
   (addm) edge [->]  (covm)
   (covm) edge [->]  (nonn)
   (nonn) edge [->]  (cfn);
\draw (addm) edge [->] (b)
      (b)  edge [->] (nonm);
\draw (covm) edge [->] (d)
      (d)  edge[->] (cfm);
\draw (b) edge [->] (d);

\draw   (none) edge [->] (nonm)
        (none) edge [sub3,->] (cfm)
        (addm) edge [->] (cove);
      
\draw (none) edge [line width=.15cm,white,-] (nonn)
      (none) edge [->] (nonn);
      
\draw (cove) edge [line width=.15cm,white,-] (covn)
      (cove) edge [<-] (covm)
      (cove) edge [<-] (covn);

\draw (addm) edge [line width=.15cm,white,-] (none)
      (addm) edge [sub3,->] (none); 

\draw (cove) edge [line width=.15cm,white,-] (cfm)
      (cove) edge [sub3,->] (cfm);

\draw (e) edge [line width=.15cm,white,-] (none)
      (e) edge [->] (none);   

\draw (e) edge [line width=.15cm,white,-] (covm)
      (e) edge [->] (covm);  

\draw (e) edge [line width=.15cm,white,-] (aleph1)
      (e) edge [->] (aleph1);

\draw[color=sug,line width=.05cm] (4.95,2.5)--(4.95,8.5);
\draw[color=sug,line width=.05cm] (0.75,6.45)--(4.95,6.45); 
\draw[color=sug,line width=.05cm] (-0.55,5.45)--(4.95,5.45);
\draw[color=sug,line width=.05cm] (-0.55,2.5)--(-0.55,8.5);
\draw[color=sug,line width=.05cm] (1.8,6.45)--(1.8,8.5);
\draw[color=sug,line width=.05cm] (0.75,5.45)--(0.75,6.45); 
\draw[color=sug,line width=.05cm] (0.65,4.3)--(0.65,5.45);
\draw[color=sug,line width=.05cm] (0.65,4.3)--(2.5,4.3);
\draw[color=sug,line width=.05cm] (1.4,2.5)--(1.4,4.3);
\draw[color=sug,line width=.05cm] (2.5,4.3)--(2.5,5.45);

\draw[circle, fill=yellow,color=yellow] (1.35,5.95) circle (0.4);
\draw[circle, fill=yellow,color=yellow] (0,4.9) circle (0.4);
\draw[circle, fill=yellow,color=yellow] (0,5.95) circle (0.4);
\draw[circle, fill=yellow,color=yellow] (2.5,3.7) circle (0.4);
\draw[circle, fill=yellow,color=yellow] (8.2,5) circle (0.4);
\draw[circle, fill=yellow,color=yellow] (3.19,7.8) circle (0.4);
\draw[circle, fill=yellow,color=yellow] (1.2,4.9) circle (0.4);
\node at (8.2,5) {$\lambda_6$};
\node at (1.35,5.95) {$\lambda_4$};
\node at (1.2,4.9) {$\lambda_3$};
\node at (0,4.9) {$\lambda_0$};
\node at (0,5.95) {$\lambda_2$};
\node at (2.5,3.7) {$\lambda_1$};
\node at (3.19,7.8) {$\lambda_5$};
}
\end{tikzpicture}
\caption{Constellation of~\autoref{prob:E}~\ref{probE:d}}
\label{FigEd}
\end{figure}

\begin{figure}[ht!]
\centering
\begin{tikzpicture}[scale=1.06]
\small{
\node (aleph1) at (-1,3) {$\aleph_1$};
\node (addn) at (0.5,3){$\add(\Nwf)$};
\node (covn) at (0.5,7){$\cov(\Nwf)$};
\node (nonn) at (9.5,3) {$\non(\Nwf)$} ;
\node (cfn) at (9.5,7) {$\cof(\Nwf)$} ;
\node (addm) at (3.19,3) {$\add(\Mwf)=\subiii{\add(\Ewf)}$} ;
\node (covm) at (6.9,3) {$\cov(\Mwf)$} ;
\node (nonm) at (3.19,7) {$\non(\Mwf)$} ;
\node (cfm) at (6.9,7) {$\cof(\Mwf)=\subiii{\cof(\Ewf)}$} ;
\node (b) at (3.19,5) {$\bfrak$};
\node (d) at (6.9,5) {$\dfrak$};
\node (c) at (11,7) {$\cfrak$};
\node (e) at (2,4.8) {$\efrak$};
\node (none) at (4.12,6) {\subiii{$\non(\Ewf)$}};
\node (cove) at (5.8,4.25) {\subiii{$\cov(\Ewf)$}};
\draw (aleph1) edge[->] (addn)
      (addn) edge[->] (covn)
      (covn) edge [->] (nonm)
      (nonm)edge [->] (cfm)
      (cfm)edge [->] (cfn)
      (cfn) edge[->] (c);

\draw
   (addn) edge [->]  (addm)
   (addm) edge [->]  (covm)
   (covm) edge [->]  (nonn)
   (nonn) edge [->]  (cfn);
\draw (addm) edge [->] (b)
      (b)  edge [->] (nonm);
\draw (covm) edge [->] (d)
      (d)  edge[->] (cfm);
\draw (b) edge [->] (d);

\draw   (none) edge [->] (nonm)
        (none) edge [sub3,->] (cfm)
        (addm) edge [->] (cove);
      
\draw (none) edge [line width=.15cm,white,-] (nonn)
      (none) edge [->] (nonn);
      
\draw (cove) edge [line width=.15cm,white,-] (covn)
      (cove) edge [<-] (covm)
      (cove) edge [<-] (covn);

\draw (addm) edge [line width=.15cm,white,-] (none)
      (addm) edge [sub3,->] (none); 

\draw (cove) edge [line width=.15cm,white,-] (cfm)
      (cove) edge [sub3,->] (cfm);

\draw (e) edge [line width=.15cm,white,-] (none)
      (e) edge [->] (none);   

\draw (e) edge [line width=.15cm,white,-] (covm)
      (e) edge [->] (covm);  

\draw (e) edge [line width=.15cm,white,-] (aleph1)
      (e) edge [->] (aleph1);

\draw[color=sug,line width=.05cm] (4.95,2.5)--(4.95,8.5);
\draw[color=sug,line width=.05cm] (0.75,6.45)--(4.95,6.45); 
\draw[color=sug,line width=.05cm] (-0.55,5.45)--(4.95,5.45);
\draw[color=sug,line width=.05cm] (-0.55,2.5)--(-0.55,8.5);
\draw[color=sug,line width=.05cm] (1.8,6.45)--(1.8,8.5);
\draw[color=sug,line width=.05cm] (0.75,5.45)--(0.75,6.45); 
\draw[color=sug,line width=.05cm] (0.65,4.3)--(0.65,5.45);
\draw[color=sug,line width=.05cm] (0.65,4.3)--(2.5,4.3);
\draw[color=sug,line width=.05cm] (1.4,2.5)--(1.4,4.3);
\draw[color=sug,line width=.05cm] (2.5,4.3)--(2.5,5.45);

\draw[circle, fill=yellow,color=yellow] (1.35,5.95) circle (0.4);
\draw[circle, fill=yellow,color=yellow] (0,4.9) circle (0.4);
\draw[circle, fill=yellow,color=yellow] (0,5.95) circle (0.4);
\draw[circle, fill=yellow,color=yellow] (2.5,3.7) circle (0.4);
\draw[circle, fill=yellow,color=yellow] (8.2,5) circle (0.4);
\draw[circle, fill=yellow,color=yellow] (3.19,7.8) circle (0.4);
\draw[circle, fill=yellow,color=yellow] (1.2,4.9) circle (0.4);
\node at (8.2,5) {$\lambda_6$};
\node at (1.35,5.95) {$\lambda_4$};
\node at (1.2,4.9) {$\lambda_2$};
\node at (0,4.9) {$\lambda_0$};
\node at (0,5.95) {$\lambda_3$};
\node at (2.5,3.7) {$\lambda_1$};
\node at (3.19,7.8) {$\lambda_5$};
}
\end{tikzpicture}
\caption{Constellation of~\autoref{prob:E}~\ref{probE:e}}
\label{FigEe}
\end{figure}

\begin{figure}[ht!]
\centering
\begin{tikzpicture}[scale=1.06]
\small{
\node (aleph1) at (-1,3) {$\aleph_1$};
\node (addn) at (0.5,3){$\add(\Nwf)$};
\node (covn) at (0.5,7){$\cov(\Nwf)$};
\node (nonn) at (9.5,3) {$\non(\Nwf)$} ;
\node (cfn) at (9.5,7) {$\cof(\Nwf)$} ;
\node (addm) at (3.19,3) {$\add(\Mwf)=\subiii{\add(\Ewf)}$} ;
\node (covm) at (6.9,3) {$\cov(\Mwf)$} ;
\node (nonm) at (3.19,7) {$\non(\Mwf)$} ;
\node (cfm) at (6.9,7) {$\cof(\Mwf)=\subiii{\cof(\Ewf)}$} ;
\node (b) at (3.19,5) {$\bfrak$};
\node (d) at (6.9,5) {$\dfrak$};
\node (c) at (11,7) {$\cfrak$};
\node (e) at (2,4.8) {$\efrak$};
\node (none) at (4.12,6) {\subiii{$\non(\Ewf)$}};
\node (cove) at (5.8,4.25) {\subiii{$\cov(\Ewf)$}};
\draw (aleph1) edge[->] (addn)
      (addn) edge[->] (covn)
      (covn) edge [->] (nonm)
      (nonm)edge [->] (cfm)
      (cfm)edge [->] (cfn)
      (cfn) edge[->] (c);

\draw
   (addn) edge [->]  (addm)
   (addm) edge [->]  (covm)
   (covm) edge [->]  (nonn)
   (nonn) edge [->]  (cfn);
\draw (addm) edge [->] (b)
      (b)  edge [->] (nonm);
\draw (covm) edge [->] (d)
      (d)  edge[->] (cfm);
\draw (b) edge [->] (d);

\draw   (none) edge [->] (nonm)
        (none) edge [sub3,->] (cfm)
        (addm) edge [->] (cove);
      
\draw (none) edge [line width=.15cm,white,-] (nonn)
      (none) edge [->] (nonn);
      
\draw (cove) edge [line width=.15cm,white,-] (covn)
      (cove) edge [<-] (covm)
      (cove) edge [<-] (covn);

\draw (addm) edge [line width=.15cm,white,-] (none)
      (addm) edge [sub3,->] (none); 

\draw (cove) edge [line width=.15cm,white,-] (cfm)
      (cove) edge [sub3,->] (cfm);

\draw (e) edge [line width=.15cm,white,-] (none)
      (e) edge [->] (none);   

\draw (e) edge [line width=.15cm,white,-] (covm)
      (e) edge [->] (covm);  

\draw (e) edge [line width=.15cm,white,-] (aleph1)
      (e) edge [->] (aleph1);

\draw[color=sug,line width=.05cm] (4.95,2.5)--(4.95,8.5);
\draw[color=sug,line width=.05cm] (0.75,6.45)--(4.95,6.45); 
\draw[color=sug,line width=.05cm] (-0.55,5.45)--(4.95,5.45);
\draw[color=sug,line width=.05cm] (-0.55,2.5)--(-0.55,8.5);
\draw[color=sug,line width=.05cm] (1.8,6.45)--(1.8,8.5);
\draw[color=sug,line width=.05cm] (0.75,5.45)--(0.75,6.45); 
\draw[color=sug,line width=.05cm] (0.65,4.3)--(0.65,5.45);
\draw[color=sug,line width=.05cm] (0.65,4.3)--(2.5,4.3);
\draw[color=sug,line width=.05cm] (1.4,2.5)--(1.4,4.3);
\draw[color=sug,line width=.05cm] (2.5,4.3)--(2.5,5.45);

\draw[circle, fill=yellow,color=yellow] (1.35,5.95) circle (0.4);
\draw[circle, fill=yellow,color=yellow] (0,4.9) circle (0.4);
\draw[circle, fill=yellow,color=yellow] (0,5.95) circle (0.4);
\draw[circle, fill=yellow,color=yellow] (2.5,3.7) circle (0.4);
\draw[circle, fill=yellow,color=yellow] (8.2,5) circle (0.4);
\draw[circle, fill=yellow,color=yellow] (3.19,7.8) circle (0.4);
\draw[circle, fill=yellow,color=yellow] (1.2,4.9) circle (0.4);
\node at (8.2,5) {$\lambda_6$};
\node at (1.35,5.95) {$\lambda_3$};
\node at (1.2,4.9) {$\lambda_2$};
\node at (0,4.9) {$\lambda_0$};
\node at (0,5.95) {$\lambda_4$};
\node at (2.5,3.7) {$\lambda_1$};
\node at (3.19,7.8) {$\lambda_5$};
}
\end{tikzpicture}
\caption{Constellation of~\autoref{prob:E}~\ref{probE:f}}
\label{FigEf}
\end{figure}

\begin{figure}[ht!]
\centering
\begin{tikzpicture}[scale=1.06]
\small{
\node (aleph1) at (-1,3) {$\aleph_1$};
\node (addn) at (0.5,3){$\add(\Nwf)$};
\node (covn) at (0.5,7){$\cov(\Nwf)$};
\node (nonn) at (9.5,3) {$\non(\Nwf)$} ;
\node (cfn) at (9.5,7) {$\cof(\Nwf)$} ;
\node (addm) at (3.19,3) {$\add(\Mwf)=\subiii{\add(\Ewf)}$} ;
\node (covm) at (6.9,3) {$\cov(\Mwf)$} ;
\node (nonm) at (3.19,7) {$\non(\Mwf)$} ;
\node (cfm) at (6.9,7) {$\cof(\Mwf)=\subiii{\cof(\Ewf)}$} ;
\node (b) at (3.19,5) {$\bfrak$};
\node (d) at (6.9,5) {$\dfrak$};
\node (c) at (11,7) {$\cfrak$};
\node (e) at (2,4.8) {$\efrak$};
\node (none) at (4.12,6) {\subiii{$\non(\Ewf)$}};
\node (cove) at (5.8,4.25) {\subiii{$\cov(\Ewf)$}};
\draw (aleph1) edge[->] (addn)
      (addn) edge[->] (covn)
      (covn) edge [->] (nonm)
      (nonm)edge [->] (cfm)
      (cfm)edge [->] (cfn)
      (cfn) edge[->] (c);

\draw
   (addn) edge [->]  (addm)
   (addm) edge [->]  (covm)
   (covm) edge [->]  (nonn)
   (nonn) edge [->]  (cfn);
\draw (addm) edge [->] (b)
      (b)  edge [->] (nonm);
\draw (covm) edge [->] (d)
      (d)  edge[->] (cfm);
\draw (b) edge [->] (d);

\draw   (none) edge [->] (nonm)
        (none) edge [sub3,->] (cfm)
        (addm) edge [->] (cove);
      
\draw (none) edge [line width=.15cm,white,-] (nonn)
      (none) edge [->] (nonn);
      
\draw (cove) edge [line width=.15cm,white,-] (covn)
      (cove) edge [<-] (covm)
      (cove) edge [<-] (covn);

\draw (addm) edge [line width=.15cm,white,-] (none)
      (addm) edge [sub3,->] (none); 

\draw (cove) edge [line width=.15cm,white,-] (cfm)
      (cove) edge [sub3,->] (cfm);

\draw (e) edge [line width=.15cm,white,-] (none)
      (e) edge [->] (none);   

\draw (e) edge [line width=.15cm,white,-] (covm)
      (e) edge [->] (covm);  

\draw (e) edge [line width=.15cm,white,-] (aleph1)
      (e) edge [->] (aleph1);

\draw[color=sug,line width=.05cm] (4.95,2.5)--(4.95,8.5);
\draw[color=sug,line width=.05cm] (0.75,6.45)--(4.95,6.45); 
\draw[color=sug,line width=.05cm] (-0.55,5.45)--(4.95,5.45);
\draw[color=sug,line width=.05cm] (-0.55,2.5)--(-0.55,8.5);
\draw[color=sug,line width=.05cm] (1.8,6.45)--(1.8,8.5);
\draw[color=sug,line width=.05cm] (0.75,5.45)--(0.75,6.45); 
\draw[color=sug,line width=.05cm] (0.65,4.3)--(0.65,5.45);
\draw[color=sug,line width=.05cm] (0.65,4.3)--(2.5,4.3);
\draw[color=sug,line width=.05cm] (1.4,2.5)--(1.4,4.3);
\draw[color=sug,line width=.05cm] (2.5,4.3)--(2.5,5.45);

\draw[circle, fill=yellow,color=yellow] (1.35,5.95) circle (0.4);
\draw[circle, fill=yellow,color=yellow] (0,4.9) circle (0.4);
\draw[circle, fill=yellow,color=yellow] (0,5.95) circle (0.4);
\draw[circle, fill=yellow,color=yellow] (2.5,3.7) circle (0.4);
\draw[circle, fill=yellow,color=yellow] (8.2,5) circle (0.4);
\draw[circle, fill=yellow,color=yellow] (3.19,7.8) circle (0.4);
\draw[circle, fill=yellow,color=yellow] (1.2,4.9) circle (0.4);
\node at (8.2,5) {$\lambda_6$};
\node at (1.35,5.95) {$\lambda_3$};
\node at (1.2,4.9) {$\lambda_1$};
\node at (0,4.9) {$\lambda_0$};
\node at (0,5.95) {$\lambda_4$};
\node at (2.5,3.7) {$\lambda_2$};
\node at (3.19,7.8) {$\lambda_5$};
}
\end{tikzpicture}
\caption{Constellation of~\autoref{prob:E}~\ref{probE:g}}
\label{FigEg}
\end{figure}

\begin{figure}[ht!]
\centering
\begin{tikzpicture}[scale=1.06]
\small{
\node (aleph1) at (-1,3) {$\aleph_1$};
\node (addn) at (0.5,3){$\add(\Nwf)$};
\node (covn) at (0.5,7){$\cov(\Nwf)$};
\node (nonn) at (9.5,3) {$\non(\Nwf)$} ;
\node (cfn) at (9.5,7) {$\cof(\Nwf)$} ;
\node (addm) at (3.19,3) {$\add(\Mwf)=\subiii{\add(\Ewf)}$} ;
\node (covm) at (6.9,3) {$\cov(\Mwf)$} ;
\node (nonm) at (3.19,7) {$\non(\Mwf)$} ;
\node (cfm) at (6.9,7) {$\cof(\Mwf)=\subiii{\cof(\Ewf)}$} ;
\node (b) at (3.19,5) {$\bfrak$};
\node (d) at (6.9,5) {$\dfrak$};
\node (c) at (11,7) {$\cfrak$};
\node (e) at (2,4.8) {$\efrak$};
\node (none) at (4.12,6) {\subiii{$\non(\Ewf)$}};
\node (cove) at (5.8,4.25) {\subiii{$\cov(\Ewf)$}};
\draw (aleph1) edge[->] (addn)
      (addn) edge[->] (covn)
      (covn) edge [->] (nonm)
      (nonm)edge [->] (cfm)
      (cfm)edge [->] (cfn)
      (cfn) edge[->] (c);

\draw
   (addn) edge [->]  (addm)
   (addm) edge [->]  (covm)
   (covm) edge [->]  (nonn)
   (nonn) edge [->]  (cfn);
\draw (addm) edge [->] (b)
      (b)  edge [->] (nonm);
\draw (covm) edge [->] (d)
      (d)  edge[->] (cfm);
\draw (b) edge [->] (d);

\draw   (none) edge [->] (nonm)
        (none) edge [sub3,->] (cfm)
        (addm) edge [->] (cove);
      
\draw (none) edge [line width=.15cm,white,-] (nonn)
      (none) edge [->] (nonn);
      
\draw (cove) edge [line width=.15cm,white,-] (covn)
      (cove) edge [<-] (covm)
      (cove) edge [<-] (covn);

\draw (addm) edge [line width=.15cm,white,-] (none)
      (addm) edge [sub3,->] (none); 

\draw (cove) edge [line width=.15cm,white,-] (cfm)
      (cove) edge [sub3,->] (cfm);

\draw (e) edge [line width=.15cm,white,-] (none)
      (e) edge [->] (none);   

\draw (e) edge [line width=.15cm,white,-] (covm)
      (e) edge [->] (covm);  

\draw (e) edge [line width=.15cm,white,-] (aleph1)
      (e) edge [->] (aleph1);

\draw[color=sug,line width=.05cm] (4.95,2.5)--(4.95,8.5);
\draw[color=sug,line width=.05cm] (0.75,6.45)--(4.95,6.45); 
\draw[color=sug,line width=.05cm] (-0.55,5.45)--(4.95,5.45);
\draw[color=sug,line width=.05cm] (-0.55,2.5)--(-0.55,8.5);
\draw[color=sug,line width=.05cm] (1.8,6.45)--(1.8,8.5);
\draw[color=sug,line width=.05cm] (0.75,5.45)--(0.75,6.45); 
\draw[color=sug,line width=.05cm] (0.65,4.3)--(0.65,5.45);
\draw[color=sug,line width=.05cm] (0.65,4.3)--(2.5,4.3);
\draw[color=sug,line width=.05cm] (1.4,2.5)--(1.4,4.3);
\draw[color=sug,line width=.05cm] (2.5,4.3)--(2.5,5.45);

\draw[circle, fill=yellow,color=yellow] (1.35,5.95) circle (0.4);
\draw[circle, fill=yellow,color=yellow] (0,4.9) circle (0.4);
\draw[circle, fill=yellow,color=yellow] (0,5.95) circle (0.4);
\draw[circle, fill=yellow,color=yellow] (2.5,3.7) circle (0.4);
\draw[circle, fill=yellow,color=yellow] (8.2,5) circle (0.4);
\draw[circle, fill=yellow,color=yellow] (3.19,7.8) circle (0.4);
\draw[circle, fill=yellow,color=yellow] (1.2,4.9) circle (0.4);
\node at (8.2,5) {$\lambda_6$};
\node at (1.35,5.95) {$\lambda_4$};
\node at (1.2,4.9) {$\lambda_1$};
\node at (0,4.9) {$\lambda_0$};
\node at (0,5.95) {$\lambda_3$};
\node at (2.5,3.7) {$\lambda_2$};
\node at (3.19,7.8) {$\lambda_5$};
}
\end{tikzpicture}
\caption{Constellation of~\autoref{prob:E}~\ref{probE:h}}
\label{FigEh}
\end{figure}

\begin{figure}[ht!]
\centering
\begin{tikzpicture}[scale=1.06]
\small{
\node (aleph1) at (-1,3) {$\aleph_1$};
\node (addn) at (0.5,3){$\add(\Nwf)$};
\node (covn) at (0.5,7){$\cov(\Nwf)$};
\node (nonn) at (9.5,3) {$\non(\Nwf)$} ;
\node (cfn) at (9.5,7) {$\cof(\Nwf)$} ;
\node (addm) at (3.19,3) {$\add(\Mwf)=\subiii{\add(\Ewf)}$} ;
\node (covm) at (6.9,3) {$\cov(\Mwf)$} ;
\node (nonm) at (3.19,7) {$\non(\Mwf)$} ;
\node (cfm) at (6.9,7) {$\cof(\Mwf)=\subiii{\cof(\Ewf)}$} ;
\node (b) at (3.19,5) {$\bfrak$};
\node (d) at (6.9,5) {$\dfrak$};
\node (c) at (11,7) {$\cfrak$};
\node (e) at (2,4.8) {$\efrak$};
\node (none) at (4.12,6) {\subiii{$\non(\Ewf)$}};
\node (cove) at (5.8,4.25) {\subiii{$\cov(\Ewf)$}};
\draw (aleph1) edge[->] (addn)
      (addn) edge[->] (covn)
      (covn) edge [->] (nonm)
      (nonm)edge [->] (cfm)
      (cfm)edge [->] (cfn)
      (cfn) edge[->] (c);

\draw
   (addn) edge [->]  (addm)
   (addm) edge [->]  (covm)
   (covm) edge [->]  (nonn)
   (nonn) edge [->]  (cfn);
\draw (addm) edge [->] (b)
      (b)  edge [->] (nonm);
\draw (covm) edge [->] (d)
      (d)  edge[->] (cfm);
\draw (b) edge [->] (d);

\draw   (none) edge [->] (nonm)
        (none) edge [sub3,->] (cfm)
        (addm) edge [->] (cove);
      
\draw (none) edge [line width=.15cm,white,-] (nonn)
      (none) edge [->] (nonn);
      
\draw (cove) edge [line width=.15cm,white,-] (covn)
      (cove) edge [<-] (covm)
      (cove) edge [<-] (covn);

\draw (addm) edge [line width=.15cm,white,-] (none)
      (addm) edge [sub3,->] (none); 

\draw (cove) edge [line width=.15cm,white,-] (cfm)
      (cove) edge [sub3,->] (cfm);

\draw (e) edge [line width=.15cm,white,-] (none)
      (e) edge [->] (none);   

\draw (e) edge [line width=.15cm,white,-] (covm)
      (e) edge [->] (covm);  

\draw (e) edge [line width=.15cm,white,-] (aleph1)
      (e) edge [->] (aleph1);

\draw[color=sug,line width=.05cm] (4.95,2.5)--(4.95,8.5);
\draw[color=sug,line width=.05cm] (0.75,6.45)--(4.95,6.45); 
\draw[color=sug,line width=.05cm] (-0.55,5.45)--(4.95,5.45);
\draw[color=sug,line width=.05cm] (-0.55,2.5)--(-0.55,8.5);
\draw[color=sug,line width=.05cm] (1.8,6.45)--(1.8,8.5);
\draw[color=sug,line width=.05cm] (0.75,5.45)--(0.75,6.45); 
\draw[color=sug,line width=.05cm] (0.65,4.3)--(0.65,5.45);
\draw[color=sug,line width=.05cm] (0.65,4.3)--(2.5,4.3);
\draw[color=sug,line width=.05cm] (1.4,2.5)--(1.4,4.3);
\draw[color=sug,line width=.05cm] (2.5,4.3)--(2.5,5.45);

\draw[circle, fill=yellow,color=yellow] (1.35,5.95) circle (0.4);
\draw[circle, fill=yellow,color=yellow] (0,4.9) circle (0.4);
\draw[circle, fill=yellow,color=yellow] (0,5.95) circle (0.4);
\draw[circle, fill=yellow,color=yellow] (2.5,3.7) circle (0.4);
\draw[circle, fill=yellow,color=yellow] (8.2,5) circle (0.4);
\draw[circle, fill=yellow,color=yellow] (3.19,7.8) circle (0.4);
\draw[circle, fill=yellow,color=yellow] (1.2,4.9) circle (0.4);
\node at (8.2,5) {$\lambda_6$};
\node at (1.35,5.95) {$\lambda_4$};
\node at (1.2,4.9) {$\lambda_1$};
\node at (0,4.9) {$\lambda_0$};
\node at (0,5.95) {$\lambda_2$};
\node at (2.5,3.7) {$\lambda_3$};
\node at (3.19,7.8) {$\lambda_5$};
}
\end{tikzpicture}
\caption{Constellation of~\autoref{prob:E}~\ref{probE:i}}
\label{FigEi}
\end{figure}

\begin{figure}[H]
\centering
\begin{tikzpicture}[scale=1.06]
\small{
\node (aleph1) at (-1,3) {$\aleph_1$};
\node (addn) at (0.5,3){$\add(\Nwf)$};
\node (covn) at (0.5,7){$\cov(\Nwf)$};
\node (nonn) at (9.5,3) {$\non(\Nwf)$} ;
\node (cfn) at (9.5,7) {$\cof(\Nwf)$} ;
\node (addm) at (3.19,3) {$\add(\Mwf)=\subiii{\add(\Ewf)}$} ;
\node (covm) at (6.9,3) {$\cov(\Mwf)$} ;
\node (nonm) at (3.19,7) {$\non(\Mwf)$} ;
\node (cfm) at (6.9,7) {$\cof(\Mwf)=\subiii{\cof(\Ewf)}$} ;
\node (b) at (3.19,5) {$\bfrak$};
\node (d) at (6.9,5) {$\dfrak$};
\node (c) at (11,7) {$\cfrak$};
\node (e) at (2,4.8) {$\efrak$};
\node (none) at (4.12,6) {\subiii{$\non(\Ewf)$}};
\node (cove) at (5.8,4.25) {\subiii{$\cov(\Ewf)$}};
\draw (aleph1) edge[->] (addn)
      (addn) edge[->] (covn)
      (covn) edge [->] (nonm)
      (nonm)edge [->] (cfm)
      (cfm)edge [->] (cfn)
      (cfn) edge[->] (c);

\draw
   (addn) edge [->]  (addm)
   (addm) edge [->]  (covm)
   (covm) edge [->]  (nonn)
   (nonn) edge [->]  (cfn);
\draw (addm) edge [->] (b)
      (b)  edge [->] (nonm);
\draw (covm) edge [->] (d)
      (d)  edge[->] (cfm);
\draw (b) edge [->] (d);

\draw   (none) edge [->] (nonm)
        (none) edge [sub3,->] (cfm)
        (addm) edge [->] (cove);
      
\draw (none) edge [line width=.15cm,white,-] (nonn)
      (none) edge [->] (nonn);
      
\draw (cove) edge [line width=.15cm,white,-] (covn)
      (cove) edge [<-] (covm)
      (cove) edge [<-] (covn);

\draw (addm) edge [line width=.15cm,white,-] (none)
      (addm) edge [sub3,->] (none); 

\draw (cove) edge [line width=.15cm,white,-] (cfm)
      (cove) edge [sub3,->] (cfm);

\draw (e) edge [line width=.15cm,white,-] (none)
      (e) edge [->] (none);   

\draw (e) edge [line width=.15cm,white,-] (covm)
      (e) edge [->] (covm);  

\draw (e) edge [line width=.15cm,white,-] (aleph1)
      (e) edge [->] (aleph1);

\draw[color=sug,line width=.05cm] (4.95,2.5)--(4.95,8.5);
\draw[color=sug,line width=.05cm] (0.75,6.45)--(4.95,6.45); 
\draw[color=sug,line width=.05cm] (-0.55,5.45)--(4.95,5.45);
\draw[color=sug,line width=.05cm] (-0.55,2.5)--(-0.55,8.5);
\draw[color=sug,line width=.05cm] (1.8,6.45)--(1.8,8.5);
\draw[color=sug,line width=.05cm] (0.75,5.45)--(0.75,6.45); 
\draw[color=sug,line width=.05cm] (0.65,4.3)--(0.65,5.45);
\draw[color=sug,line width=.05cm] (0.65,4.3)--(2.5,4.3);
\draw[color=sug,line width=.05cm] (1.4,2.5)--(1.4,4.3);
\draw[color=sug,line width=.05cm] (2.5,4.3)--(2.5,5.45);

\draw[circle, fill=yellow,color=yellow] (1.35,5.95) circle (0.4);
\draw[circle, fill=yellow,color=yellow] (0,4.9) circle (0.4);
\draw[circle, fill=yellow,color=yellow] (0,5.95) circle (0.4);
\draw[circle, fill=yellow,color=yellow] (2.5,3.7) circle (0.4);
\draw[circle, fill=yellow,color=yellow] (8.2,5) circle (0.4);
\draw[circle, fill=yellow,color=yellow] (3.19,7.8) circle (0.4);
\draw[circle, fill=yellow,color=yellow] (1.2,4.9) circle (0.4);
\node at (8.2,5) {$\lambda_6$};
\node at (1.35,5.95) {$\lambda_4$};
\node at (1.2,4.9) {$\lambda_2$};
\node at (0,4.9) {$\lambda_0$};
\node at (0,5.95) {$\lambda_1$};
\node at (2.5,3.7) {$\lambda_3$};
\node at (3.19,7.8) {$\lambda_5$};
}
\end{tikzpicture}
\caption{Constellation of~\autoref{prob:E}~\ref{probE:j}}
\label{FigEj}
\end{figure}

\subsection*{Acknowledgments}

This note has been developed specifically for the proceedings of the RIMS Set Theory Workshop 2023 \textit{Large Cardinals and the Continuu}, held at Kyoto University RIMS. The author expresses gratitude to Professor Hiroshi Fujita from Ehime University for letting him participate with
a talk at the Workshop and submit a paper to this proceedings.

The author would also like to thank: Diego A. Mej\'ia, Adam Marton, and Jaroslav \v{S}upina, for revision and suggestions to this paper, which helped us deliver a more presentable final version.

This work was partially supported by the Slovak Research and Development Agency under Contract No.~APVV-20-0045 and by Pavol Jozef \v{S}af\'arik University in Ko\v{s}ice at a postdoctoral position.

{\small
\bibliography{left}
\bibliographystyle{alpha}
}

\end{document}